\documentclass[12pt]{amsart}   

\usepackage{amssymb}
\usepackage{graphicx}
\usepackage{amsmath}
\usepackage{amscd}
\usepackage{latexsym}
\usepackage{tikz}
\usetikzlibrary{decorations.pathreplacing}
\usepackage{siunitx}
\usepackage{float}

\usepackage{amsmath}
\usepackage{MnSymbol}
\usepackage{wasysym}

\newtheorem{thm}{Theorem}

\newtheorem{defn}[thm]{Definition}
\newtheorem{rmk}[thm]{Remark}
\newtheorem{cor}[thm]{Corollary}
\newtheorem{examp}[thm]{Example}
\newtheorem{prop}[thm]{Proposition}
\newtheorem{problem}{Problem}

\DeclareMathOperator{\Aye}{I}

\DeclareMathOperator{\I}{I}
\DeclareMathOperator{\Lat}{Lat}
\DeclareMathOperator{\Rel}{Rel}
\DeclareMathOperator{\Conv}{Conv}
\DeclareMathOperator{\Alg}{Alg}
\DeclareMathOperator{\uu}{u}

\begin{document}

\title[The convolution algebra]{The convolution algebra}

\author[Harding]{John Harding, Carol Walker, and Elbert Walker}
\address{Department of Mathematical Sciences\\
New Mexico State University\\
Las Cruces, NM 88003, USA}

\email{jharding@nmsu.edu, hardy@nmsu.edu, elbert@nmsu.edu\phantom{,}}

\dedicatory{Dedicated to the memory of Bjarni J\'{o}nsson}

\begin{abstract}
For a complete lattice $L$ and a relational structure $\mathfrak{X}=(X,(R_i)_I)$, we introduce the convolution algebra $L^{\mathfrak{X}}$. This algebra consists of the lattice $L^X$ equipped with an additional $n_i$-ary operation $f_i$ for each $n_i+1$-ary relation $R_i$ of $\mathfrak{X}$. For $\alpha_1,\ldots,\alpha_{n_i}\in L^X$ and $x\in X$ we set $f_i(\alpha_1,\ldots,\alpha_{n_i})(x)=\bigvee\{\alpha_1(x_1)\wedge\cdots\wedge\alpha_{n_i}(x_{n_i}):(x_1,\ldots,x_{n_i},x)\in R_i\}$. For the 2-element lattice $2$, $2^\mathfrak{X}$ is the reduct of the familiar complex algebra $\mathfrak{X}^+$ obtained by removing Boolean complementation from the signature. It is shown that this construction is bifunctorial and behaves well with respect to one-one and onto maps and with respect to products. When $L$ is the reduct of a complete Heyting algebra, the operations of $L^\mathfrak{X}$ are completely additive in each coordinate and $L^\mathfrak{X}$ is in the variety generated by $2^\mathfrak{X}$. Extensions to the construction are made to allow for completely multiplicative operations defined through meets instead of joins, as well as modifications to allow for convolutions of relational structures with partial orderings. Several examples are given. 
\end{abstract}



\maketitle

\section{Introduction}

For a group $G$, its complex algebra $G^+$ is obtained by defining a multiplication on the power set of $G$ by setting $A\cdot B = \{ab:a\in A, b\in B\}$. Complex algebras of groups were introduced by Frobenius early in the development of group theory, and they play an essential role in the development of Tarski's \cite{Tarski} relation algebras. 

As part of Tarski's relation algebra program, J\'{o}nsson and Tarski \cite{J-T 1,J-T 2} extended the notion of complex algebras to apply to any relational structure $\mathfrak{X}=(X,(R_i)_I)$, that is, any set $X$ with family of relations $R_i$ $(i\in I)$. For such $\mathfrak{X}$, if for $i\in I$ the relation $R_i$ is $n_i+1$-ary, then relational image provides an $n_i$-ary operation $f_i$ on the power set of $X$. The complex algebra $\mathfrak{X}^+$ of this relational structure is the Boolean algebra of subsets of $X$ with additional operations $f_i$ $(i\in I)$. These operations are additive in each component, a property expressed by saying that they are operators. The complex algebras $\mathfrak{X}^+$ are thus primary examples of what are known as Boolean algebras with operators. 

Complex algebras of relational structures were reintroduced into modal logic by Kripke \cite{Kripke}. A relational structure $\mathfrak{X}=(X,(R_i)_I)$ with a single binary relation is a Kripke frame, the elements of $X$ are called possible worlds, and the binary relation is known as an accessibility relation. The complex algebra $\mathfrak{X}^+$ is a normal modal algebra. It is common in logical circles to define what is known as the $\Box$ operation on the power set of $X$ rather than the operation known as $\Diamond$ that is the one obtained through relational image by J\'{o}nsson and Tarski. But this is a matter of taste since $\Box$ and $\Diamond$ are inter-definable via the Boolean algebra operations. An important aspect of the use of Kripke frames in modal logic is the so-called correspondence theory. This relates first order properties of a Kripke frame $\mathfrak{X}=(X,R)$, such as transitivity, to equational properties of its complex algebra $\mathfrak{X}^+$. \pagebreak[3]

We may view the power set of a set $X$ as the set of functions from $X$ into the 2-element Boolean algebra $2=\{0,1\}$. So the complex algebra $\mathfrak{X}^+$ of a relational structure $\mathfrak{X}$ can be viewed as the Boolean algebra $2^X$ equipped with a family of additional operations. If we ignore the Boolean complementation, we obtain from the 2-element lattice, a complete lattice $2^X$ equipped with additional operations. 

This process can be generalized. For a relational structure $\mathfrak{X}=(X,(R_i)_I)$ and a complete lattice $L$, we define the convolution algebra $L^\mathfrak{X}=(L^X,(f_i)_I)$ of $\mathfrak{X}$ over $L$ as follows. The set $L^X$ of all functions from $X$ to $L$ is a complete lattice with the componentwise operations. For each $i\in I$ we use the $n_i+1$-ary relation $R_i$ to define an $n_i$-ary operation $f_i$ on this lattice where 
\begin{equation}
\tag{1}
f_i(\alpha_1,\ldots,\alpha_{n_i})(x)=\bigvee\{\alpha_1(x_1)\wedge\cdots\wedge\alpha_{n_i}(x_{n_i}):R_i(x_1,\ldots,x_{n_i},x)\}
\end{equation}
This algebra $L^\mathfrak{X}$ is called the convolution algebra since the operations $f_i$ are convolutions of the relations $R_i$ in much the same way polynomial multiplication as a sum of certain products is a convolution. For details see \cite[Defn.~1.3.3]{HWW}. 

The convolution algebra $L^\mathfrak{X}$ is a generalization of the complex algebra $2^\mathfrak{X}$. It is also a generalization of a construction from fuzzy set theory introduced by Zadeh \cite{Zadeh1,Zadeh2}. The real unit interval $\Aye=[0,1]$ with the binary operations $\max,\min$ and the unary operation $\neg x = 1-x$ of relative negation can be considered as a relational structure $\mathfrak{X}=(\Aye,\max,\min,\neg)$ with two ternary relations and one binary relation. The real unit interval $\Aye$ is also a complete lattice. With remarkable foresight, Zadeh defined the truth value algebra of type-2 fuzzy sets to be what we term here the convolution algebra $\Aye^\mathfrak{X}$. A detailed study of this algebra is found in \cite{HWW}. 

Another relative of the convolution algebra is found in Foster's work on bounded Boolean powers \cite{Ban,FosterI,FosterII}. In our terminology, Foster considered an algebra $A$ as a relational structure, and for a Boolean algebra $B$ considered all functions $\alpha:A\to B$ whose image is a finite partition of unity. Operations were defined as in (1) with the resulting algebra denoted $B[A]^*$. Foster did not need completeness of $B$ due to his restrictions to specialized functions. J\'{o}nsson \cite{Jonsson} later recognized that Foster's construction was simply the algebra of continuous functions from the Stone space of $B$ into $A$ with the discrete topology. 

Construction of the convolution algebra is bifunctorial. Let $\Lat$ be the category of complete lattices with morphisms being maps that preserve binary meets and arbitrary joins, let $\Rel_\tau$ be the class of all relational structures of a given type $\tau$ with morphisms being $p$-morphisms, and let $\Alg_{\tau}$ be the category of algebras of type $\tau$ with morphisms being homomorphisms. Then there is a bifunctor $\Conv:\Lat\times\Rel_\tau\to\Alg_{\tau}$ that is covariant in its first argument, and contravariant in its second argument, that acts on objects by taking $(L,\mathfrak{X})$ to the convolution algebra $L^\mathfrak{X}$. 

The strongest properties of convolution algebras are obtained when $L$ is well behaved. Complete, meet-continuous distributive lattices are ones that satisfy $x\wedge\bigvee_Jy_j=\bigvee_J(x\wedge y_j)$. They are exactly the lattice reducts of complete Heyting algebras. For $L$ the lattice reduct of a complete Heyting algebra, the operations of a convolution algebra $L^\mathfrak{X}$ are completely additive in each argument, so are complete operators. So the study of such convolution algebras fits with the setting of Boolean algebras with operators, or more generally, bounded distributive lattices with operators \cite{MaiBjarni}. In this setting, the negation-free fragment of the classical correspondence theory carries through intact. In particular, we show the following. 
\vspace{2ex}

\noindent {\bf Theorem. } {\em If $L$ is the lattice reduct of a complete Heyting algebra with two or more elements, then $L^\mathfrak{X}$ and $2^\mathfrak{X}$ satisfy the same equations for each relational structure $\mathfrak{X}$.}
\vspace{2ex}


Various modifications and extensions to these results are given. For a frame $\mathfrak{X}=(X,(R_i)_I)$, rather than defining operations $f_i$ on $L^X$ via (1), we can interchange the roles of meets and joins and define operations $g_i$ on $L^X$ for each $i\in I$ by setting 
\vspace{-2ex}

\begin{equation}
\tag{2}
g_i(\alpha_1,\ldots,\alpha_{n_i})(x)=\bigwedge\{\alpha_1(x_1)\vee\cdots\vee\alpha_{n_i}(x_{n_i}):R_i(x_1,\ldots,x_{n_i},x)\}
\end{equation}
\vspace{-1ex}

Let $L^{\mathfrak{X}-}$ be the collection of all functions from $X$ to $L$ with operations $g_i$ defined from the relations $R_i$ of $\mathfrak{X}$ via (2). As noted above, the algebra $2^{\mathfrak{X}}$ with the operations $f_i$ of (1) is isomorphic to the complex algebra of $\mathfrak{X}$ with the modal operators $\Diamond_i$ obtained from the relations $R_i$ for $i\in I$. The algebra $2^{\mathfrak{X}-}$ with the operations $g_i$ of (2) is isomorphic to the complex algebra of $\mathfrak{X}$ with the dual operators $\Box_i=\neg\Diamond_i\neg$ for each $i\in I$. One may include both families of operations $f_i$ and $g_i$ for $i\in I$, and we denote the resulting algebra $L^{\mathfrak{X}*}$.

The symmetry between (1) and (2) yields results when $L$ is a complete join-continuous distributive lattice, i.e. the dual of a complete Heyting algebra. In this case, the convolution algebra $L^{\mathfrak{X}-}$ satisfies all equations satisfied by $2^{\mathfrak{X}-}$. Thus when $L$ is a complete Heyting algebra, $L^\mathfrak{X}$ satisfies the same equations as the complex algebra $\mathfrak{X}^+$ in the signature $\wedge,\vee,0,1,(\Diamond_i)_I$, and when $L$ is a complete dual Heyting algebra, $L^{\mathfrak{X}-}$ satisfies the same equations as the complex algebra $\mathfrak{X}^-$ in the signature $\wedge,\vee,0,1,(\Box_i)_I$. 

When $L$ is a complete Heyting algebra and a complete dual Heyting algebra, $L^{\mathfrak{X}*}$ satisfies the same equations as the complex algebra in the signature $\wedge,\vee,0,1,(\Diamond_i)_I,(\Box_i)_I$ provided that the equation does not involve both a $\Diamond_i$ and $\Box_j$ for some $i,j\in I$. Under the stronger assumption that $L$ is complete and completely distributive, we show that $L^{\mathfrak{X}*}$ satisfies all equations of the complex algebra in the signature $\wedge,\vee,0,1,(\Diamond_i)_I,(\Box_i)_I$. This applies in particular when $L$ is a complete chain, or a finite distributive lattice. 

This paper is arranged in the following way. The second section provides the basic definitions and results. The third section provides the bifunctoriality of the convolution algebra construction and related matters. The fourth section provides results regarding preservations of equations and correspondence theory. The fifth section provides various generalizations of the convolution construction. This includes such features as a version for dual operators such as the modal $\Box$ operations, and versions corresponding to complex algebras of ordered relational structures. The final section contains several examples. 

\section{Basic definitions and properties}

For any natural number $n\geq 0$, an $n$-ary relation on a set $X$ is a subset $R\subseteq X^n$, and an $n$-ary operation on $X$ is a function $f:X^n\to X$. This includes the case $n=0$. A nullary relation on $X$ is either $\emptyset$ or $\{\emptyset\}$, which are respectively interpreted as false and true, and a nullary operation $f:X^0\to X$ is determined by its value on the sole element $\emptyset$ of $X^0$, and is often written $f(\,\,) = x$ and is interpreted as a constant in $X$. While nullary relations are perfectly well defined, they will play no role in this paper. Indeed, we will consider only $n+1$-ary relations on $X$ for $n\geq 0$, and will produce from these $n$-ary operations on $L^X$. 

\begin{defn}
A type over a set $I$ is a function $\tau:I\to \mathbb{N}$ from a set $I$ into the natural numbers. 
\end{defn}

Throughout the remainder of the paper, we will assume that we have a fixed type $\tau$ over a set $I$, and for each $i\in I$, denote $\tau(i)=n_i$. 

\begin{defn}
A relational structure of type $\tau$ is a pair $\mathfrak{X}=(X,(R_i)_I)$ consisting of a set $X$ and for each $i\in I$ an $n_i+1$-ary relation $R_i$ on $X$. An algebra of type $\tau$ is a pair $\mathcal{A}=(A,(f_i)_I)$ consisting of a set $A$ and for each $i\in I$ an $n_i$-ary operation $f_i$ on $A$. 
\end{defn}

An $n$-ary operation $f$ on a set $X$ is an $n+1$-ary relation on $X$ where the $n+1$-tuple $(x_1,\ldots,x_{n+1})$ belongs to the relation iff $f(x_1,\ldots,x_n)=x_{n+1}$. Of course not every $n+1$-ary relation comes about in this way from an $n$-ary operation. This discussion shows that an algebra of type $\tau$ is literally a relational structure of type $\tau$, although not conversely. 

\begin{defn} 
Given a relational structure $\mathfrak{X}=(X,(R_i)_i)$ of type $\tau$ and a complete lattice $L$, the convolution algebra of $\mathfrak{X}$ over $L$ is the algebra $L^\mathfrak{X}=(L^X,(f_i)_I)$ of type $\tau$, where for $i\in I$,  $\alpha_1,\ldots,\alpha_{n_i}\in L^X$ and $x\in X$
\vspace{-1ex}

\[ f_i(\alpha_1,\ldots,\alpha_{n_i})(x)\,\,=\,\,\bigvee\{\alpha_1(x_1)\wedge\cdots\wedge\alpha_{n_i}(x_{n_i}):(x_1,\ldots,x_{n_i},x)\in R_i\}\]
\end{defn} 

\begin{rmk}{\em 
For a relational structure $\mathfrak{X}=(X,(R_i)_I)$ of type $\tau$ and complete lattice $L$, each $n_i+1$-ary relation $R_i$ produces an $n_i$-ary operation $f_i$ on $L^X$. So $L^\mathfrak{X}$ is indeed an algebra of type $\tau$. However, there are additional operations of binary meet and join on $L$, and also nullary operations of constants $0$ and $1$ on $L$, and these lift to componentwise operations on $L^X$. So we can consider $L^\mathfrak{X}$ to also have bounded lattice structure in addition to its operations $(f_i)_I$. Later, it will often be the case that $L$ will have the further structure of a Heyting algebra, and then this Heyting structure will also lift componentwise to $L^\mathfrak{X}$. We choose to treat the basic type of $L^\mathfrak{X}$ as $\tau$, and discuss the additional componentwise structure as the situation dictates. 
}
\end{rmk}

While the convolution algebra has a predecessor in Zadeh's algebra of type-2 fuzzy sets \cite{HWW,Zadeh1,Zadeh2}, the general definition of the convolution algebra of a relational structure over a complete lattice seems new to this paper. The following definition of the complex algebra of a relational structure has a long history. It has its origin in the complex algebra of a group, but obtained its general form in a series of papers of J\'{o}nsson and Tarski \cite{J-T 1,J-T 2}. We note that there is a closely related notion in modal logic that gives a different version of the operations obtained. These two approaches are interdefinable, and we find the original approach of J\'{o}nsson and Tarski simpler to work with in the current setting. 

\begin{defn}
\label{complex}
Let $\mathfrak{X}=(X,(R_i)_I)$ be a relational structure of type $\tau$, and let $\mathcal{P}(X)$ be the power set of $X$.  Define an algebra $\mathfrak{X}^+=(\mathcal{P}(X),(g_i)_I)$ of type $\tau$, called the complex algebra of $\mathfrak{X}$, by setting for each $i\in I$ and each family of subsets $A_1,\ldots,A_{n_i}\subseteq X$ 
\vspace{-1.5ex}

\[ g_i(A_1,\ldots,A_{n_i})=\{x:\mbox{there exist }x_1\in A_1,\ldots, x_{n_i}\in A_{n_i}\mbox{ with } (x_1,\ldots,x_{n_i},x)\in R_i\}\]
\end{defn}

\begin{prop}
\label{curr}
For a relational structure $\mathfrak{X}$, the Boolean algebra isomorphism $\phi:2^X\to \mathcal{P}(X)$ given by $\phi(\alpha)=\{x:\alpha(x)=1\}$ is an isomorphism from the convolution algebra $2^\mathfrak{X}$ to the complex algebra $\mathfrak{X}^+$.
\end{prop}

\begin{proof}
We must show that if $i\in I$, then $\phi(f_i(\alpha_1,\ldots,\alpha_{n_i}))=g_i(\phi(\alpha_1),\ldots,\phi(\alpha_{n_i}))$. Then for $x\in X$, making use of the fact that in the Boolean algebra $2$ a join is equal to 1 if and only if one of the joinends is equal to 1, we have the following. 
\vspace{-1.5ex}

\begin{align*}
x\in\phi(f_i(\alpha_1,\ldots,\alpha_{n_i})) & \mbox{ iff } f_i(\alpha_1,\ldots,\alpha_{n_i})(x) = 1\\
&\mbox{ iff } \alpha_1(x_1)\wedge\cdots\wedge\alpha_n(x_{n_i})=1\mbox{ for some } (x_1,\ldots,x_{n_i},x)\in R_i\\
&\mbox{ iff } \alpha_1(x_1)=1,\ldots,\alpha_n(x_{n_i})=1\mbox{ for some } (x_1,\ldots,x_{n_i},x)\in R_i\\
&\mbox{ iff } \mbox{ there are }x_1\in\phi(\alpha_1),\ldots,x_{n_i}\in\phi(\alpha_{n_i})\mbox{ with }(x_1,\ldots,x_{n_i},x)\in R_i\\
&\mbox{ iff } x\in g_i(\phi(\alpha_1),\ldots,\phi(\alpha_{n_i}))
\end{align*}
\vspace{-1ex}

\noindent Since the meet in 2 of $\emptyset$ is 1, the above reasoning holds also for the case when $n_i=0$. 
\end{proof}

Many of the stronger results about convolution algebras $L^\mathfrak{X}$ are restricted to the setting where $L$ is a complete meet-continuous distributive lattice, meaning that it satisfies $x\wedge\bigvee_Jy_j=\bigvee_J(x\wedge y_j)$. Such $L$ are exactly the lattice reducts of complete Heyting algebras. 

\begin{defn}
\label{defnoperator}
An $n$-ary operation $f$ on a lattice $L$ is additive in its $k^{th}$ component if for each finite family $(y_j)_J$ in $L$ and each $x_1,\ldots,x_{k-1},x_{k+1},\ldots,x_n$ we have 

\[ f(x_1,\ldots,x_{k-1},\bigvee_Jy_j,x_{k+1},\ldots,x_n)\,\,=\,\,\bigvee_Jf(x_1,\ldots,x_{k-1},y_j,x_{k+1},\ldots,x_n) \]

\noindent In a complete lattice, an operation is completely additive in the $k^{th}$ component if the same holds for an arbitrary family $(y_j)_J$. Finally, $f$ is called an operator if it is additive in each component, and a complete operator if it is completely additive in each component. 
\end{defn}

The subject of Boolean algebras with operators was initiated by J\'{o}nsson and Tarski in \cite{J-T 1,J-T 2}. A treatment of distributive lattices with operators is found in \cite{MaiBjarni}. 

\begin{prop}
\label{operators}
Let $L$ be a complete distributive lattice and $\mathfrak{X}$ be a relational structure. Then the operations $(f_i)_I$ of the convolution algebra $L^\mathfrak{X}$ and the binary join and meet operations and nullary bounds of $L^\mathfrak{X}$ are operators. If $L$ is the lattice reduct of a complete Heyting algebra, then these operations are complete operators. 
\end{prop}

\begin{proof}
We establish the result for complete operators when $L$ is the lattice reduct of a complete Heyting algebra. This uses the infinite distributive law $a\wedge\bigvee_Jb_j=\bigvee_Ja\wedge b_j$ that holds in a complete Heyting algebra. The argument that these operations are finitely additive when $L$ is a complete distributive lattice is identical except that it uses only the finitary version of this law when $J$ is a finite set, and this holds in any distributive lattice. 

Let $i\in I$. We will show that $f_i$ is completely additive in its first component, the argument for the $k^{th}$ component is identical. Let $(\beta_j)_J$ be a family in $L^X$ and $\alpha_2,\ldots,\alpha_{n_i}\in L^X$. We must show that 
\[ f_i(\bigvee_J\beta_j,\alpha_2,\ldots,\alpha_{n_i})\,\,=\,\,\bigvee_J f_i(\beta_j,\alpha_2,\ldots,\alpha_{n_i})\]

\noindent Both are members of $L^X$, hence are functions from $X$ to $L$. So to show that they are equal, it suffices to show that their evaluations at an element $x\in X$ are equal. Using the definition of $f_i$, the fact that joins in $L^X$ are computed componentwise, and meet-continuity, we have the following. 

\begin{align*}
f_i(\bigvee_J\beta_j,\alpha_2,\ldots,\alpha_{n_i})(x) &\,\,=\,\, \bigvee\{(\bigvee_J\beta_j)(x_1)\wedge\alpha_2(x_2)\wedge\cdots\wedge\alpha_{n_i}(x_{n_i}):(x_1,\ldots,x_{n_i},x)\in R_i\}\\
&\,\,=\,\,\bigvee\{\bigvee_J(\beta_j(x_1)\wedge\alpha_2(x_2)\wedge\cdots\wedge\alpha_{n_i}(x_{n_i})):(x_1,\ldots,x_{n_i},x)\in R_i\}\\
&\,\,=\,\, \bigvee_J\bigvee\{(\beta_j(x_1)\wedge\alpha_2(x_2)\wedge\cdots\wedge\alpha_{n_i}(x_{n_i})):(x_1,\ldots,x_{n_i},x)\in R_i\}\\
&\,\,=\,\, \bigvee_J f_i(\beta_j,\alpha_2,\ldots,\alpha_{n_i})
\end{align*}
\vspace{0ex}

\noindent Thus $f_i$ is a complete operator. The binary join operation is a complete operator in any complete lattice, and the binary meet operation being a complete operator is a consequence of being meet-continuous and distributive. Nullary operations are vacuously operators. 
\end{proof}

A final property of convolution algebras is somewhat special to this topic, and will be applied later in the paper to obtain results on preservation of equations. 

\begin{defn}
\label{gwot}
A function $\delta\in L^X$ has finite support if $\{x:\delta(x)\neq 0\}$ is finite. An $n$-ary operation $f$ on $L^X$ is finitely supported if for each $\alpha_1,\ldots,\alpha_n\in L^X$
\vspace{-1ex}

\begin{equation*} 
f(\alpha_1,\ldots,\alpha_n)\,\,=\,\,\bigvee\{f(\delta_1,\ldots,\delta_n):\delta_i\leq\alpha_i\mbox{ and $\delta_i$ has finite support for each }i\leq n\}
\end{equation*}
\end{defn}
\vspace{1ex}

\begin{prop}
\label{finitely supported}
For $L$ a complete lattice and $\mathfrak{X}$ a relational structure, the operations $(f_i)_I$ of $L^\mathfrak{X}$ as well as the binary join and meet operations and the nullary bounds are finitely supported. 
\end{prop}

\begin{proof}
That the binary meet and join operations of $L^\mathfrak{X}$ are finitely supported follows from the fact that they are defined componentwise. That the bounds are finitely supported is trivial since they have no arguments. Suppose $i\in I$. Since $f_i$ is order preserving, it follows that 
\vspace{-2ex}

\begin{equation} \label{eqa}
f_i(\alpha_1,\ldots,\alpha_{n_i})\,\,\geq\,\,\bigvee\{f_i(\delta_1,\ldots,\delta_{n_i}):\delta_j\leq\alpha_j\mbox{ and $\delta_j$ has finite support for each }j\leq n_i\}
\end{equation}
\vspace{-1ex}

To show the other inequality, let $x\in X$. Then 
\vspace{-1ex}

\begin{equation} \label{eqb}
f_i(\alpha_1,\ldots,\alpha_{n_i})(x) \,\,=\,\, \bigvee\{\alpha_1(x_1)\wedge\cdots\wedge\alpha_{n_i}(x_{n_i}):(x_1,\ldots,x_{n_i},x)\in R_i\}
\end{equation}
\vspace{-1ex}

\noindent For $(x_1,\ldots,x_{n_i},x)\in R_i$ and $1\leq k\leq n_i$, let $\delta_k$ be the function that takes the same value as $\alpha_k$ at $x_k$ and is zero otherwise. Then $\delta_k\leq\alpha_k$ and is finitely supported. Further,  
\vspace{-1ex}

\[ \delta_1(x_1)\wedge\cdots\wedge\delta_{n_i}(x_{n_i})\,\,=\,\,\alpha_1(x_1)\wedge\cdots\wedge\alpha_{n_i}(x_{n_i})\] 
\vspace{-1ex}

\noindent It follows from the definition of $f_i(\delta_1,\ldots,\delta_{n_i})(x)$ that 
\vspace{-1ex}

\[ \alpha_1(x_1)\wedge\cdots\wedge\alpha_{n_i}(x_{n_i}) \,\,\leq\,\, f_i(\delta_1,\ldots,\delta_{n_i})(x) \]
\vspace{-1ex}

\noindent It follows from this and (\ref{eqb}) that when the left hand side of (\ref{eqa}) is evaluated at $x$, the result is less than or equal to the right hand side of (\ref{eqa}) evaluated at $x$. Since this is true for all $x\in X$, we obtain the other inequality in (\ref{eqb}), hence equality. 
\end{proof}

\section{Categorical aspects}

Throughout this section we assume that we are given a type $\tau$ over a set $I$. 

\begin{defn}
Let $\Lat$ be the category whose objects are complete lattices and whose morphisms are those maps between complete lattices that preserve bounds, binary meets and arbitrary joins. 
\end{defn}

We seek categorical results concerning the the construction $L^\mathfrak{X}$ for a complete lattice $L$ and relational structure $\mathfrak{X}$ of type $\tau$. There are many slightly different versions of these results, depending on the properties required of the morphisms between complete lattices $L$ and $M$. We establish results for one natural path below, and describe separately the modifications to other closely related situations. 

\begin{defn}
Let $\Alg_\tau$ be the category whose objects are complete lattices with a family of additional operations $(f_i)_I$ of type $\tau$ and whose morphisms are the homomorphisms between these bounded lattices with additional operations that preserve binary meets and joins. 
\end{defn}

\begin{prop}
\label{conv}
For a relational structure $\mathfrak{X}$ of type $\tau$, there is a functor 
\vspace{-1ex}

\[\Conv(\,\cdot\,,\mathfrak{X}):\Lat\to\Alg_\tau\]
\vspace{-1ex}

\noindent that takes a complete lattice $L$ to the convolution algebra $L^\mathfrak{X}$ and takes a morphism $\phi:L\to M$ to the homomorphism $\phi^\mathfrak{X}:L^\mathfrak{X}\to M^\mathfrak{X}$ given by $\phi^\mathfrak{X}(\alpha)=\phi\circ\alpha$. 
\end{prop}

\begin{proof}
The assignment on objects and morphisms is well defined, yields objects of $\Alg_\tau$, and preserves composition. It remains to show for a morphism $\phi$, that $\phi^\mathfrak{X}$ is a homomorphism. 

The lower bound of $L^\mathfrak{X}$ is the function $0_{L^\mathfrak{X}}$ that takes the value $0_L$ for each $x\in X$. Since a morphism $\phi:L\to M$ in $\Lat$ preserves bounds, $\phi^\mathfrak{X}(0_{L^\mathfrak{X}})(x) = \phi (0_L)= 0_M$ for each $x\in X$. So $\phi^\mathfrak{X}(0_{L^\mathfrak{X}})=0_{M^\mathfrak{X}}$. Similarly, $\phi^\mathfrak{X}$ preserves the upper bound. For a family of functions $(\alpha_j)_J$ in $L^X$, the join of this family in $L^\mathfrak{X}$ is computed componentwise. So $\phi^\mathfrak{X}(\bigvee_J\alpha_j)(x)=\phi(\bigvee_J\alpha_j(x))$. Since $\phi$ preserves arbitrary joins and joins in $M^\mathfrak{X}$ are componentwise, this equals $(\bigvee_J\phi\circ\alpha_j)(x)$. Thus $\phi^\mathfrak{X}(\bigvee_J\alpha_j) = \bigvee_J\phi^\mathfrak{X}(\alpha_j)$. A similar argument shows that $\phi^\mathfrak{X}$ preserves binary meets. 

Suppose that $i\in I$, and let $f_i$ be the additional $n_i$-ary operation on $L^\mathfrak{X}$ and $g_i$ that on $M^\mathfrak{X}$. For $\alpha_1,\ldots,\alpha_{n_i}\in L^X$ we must show that 
\vspace{-1ex}

\[\phi^\mathfrak{X} (f_i(\alpha_1,\ldots,\alpha_{n_i})) \,\,=\,\, g_i(\phi^\mathfrak{X}(\alpha_1),\ldots,\phi^\mathfrak{X}(\alpha_{n_i}))\]
\vspace{-1ex}

\noindent Both sides of this equation are functions from $X$ to $M$, so to show equality, it suffices to show equality of the sides when evaluated at some $x\in X$. Using the definition of $\phi^\mathfrak{X}$, the definition of $f_i$, and that $\phi$ preserves arbitrary joins and finite meets, we have the following. 
\vspace{-1ex}

\[\phi^\mathfrak{X}(f_i(\alpha_1,\ldots,\alpha_{n_i}))(x) \,\, = \,\, \bigvee\{\phi(\alpha_1(x_1))\wedge\cdots\wedge\phi(\alpha_{n_i}(x_{n_i})):(x_1,\ldots,x_{n_i},x)\in R_i\}\]
\vspace{-1ex}

\noindent This is equal to $g_i(\phi^\mathfrak{X}(\alpha_1),\ldots,\phi^\mathfrak{X}(\alpha_n))(x)$. 
\end{proof}

\begin{rmk}{\em 
There are many modifications that can be made to this result. The objects obtained in the image of the functor $\Conv(\,\cdot\,,\mathfrak{X})$ are not only bounded lattices with additional operations of type $\tau$, but are complete lattices. For a map $\phi:L\to M$, we require that $\phi$ preserve arbitrary joins and finite meets to obtain that $\phi^\mathfrak{X}$ is compatible with the operations $f_i$ $(i\in I)$. Otherwise, properties of $\phi$ transfer directly to properties of $\phi^\mathfrak{X}:L^\mathfrak{X}\to M^\mathfrak{X}$. If $\phi$ preserves bounds, so does $\phi^\mathfrak{X}$, and if $\phi$ preserves arbitrary meets, so does $\phi^\mathfrak{X}$. This is a consequence of the fact that these operations are coordinatewise in both $L^\mathfrak{X}$ and $M^\mathfrak{X}$. 
}
\end{rmk}

\begin{prop}
\label{prod}
Let $\mathfrak{X}$ be a relational structure of type $\tau$ and $(L_j)_J$ a family of complete lattices. Then there is an isomorphism that preserves the bounded lattice operations and the additional operations of type $\tau$ given by $\Phi:(\prod_JL_j)^\mathfrak{X}\to\prod_J(L_j^\mathfrak{X})$ where
\vspace{-1ex}

\[ \Phi(\alpha)(j)(x) \,\, = \,\,\alpha(x)(j) \]
\end{prop}

\begin{proof}
It is routine that $\Phi$ is an isomorphism of bounded lattices. Suppose $i\in I$. We use $f_i$ for the operation of $(\prod_JL_j)^\mathfrak{X}$, for $j\in J$ we use $f_j^i$ for the operation of $L_j^\mathfrak{X}$, and we use $g_i$ for the operation of $\prod_JL_j^\mathfrak{X}$. Let $\alpha_1,\ldots,\alpha_{n_i}\in (\prod_JL_j)^\mathfrak{X}$, $x\in \mathfrak{X}$, and $j\in J$. Then \vspace{-1ex}

\[f_i(\alpha_1,\ldots,\alpha_{n_i})(x) \,\,=\,\, \bigvee\{\alpha_1(x_1)\wedge\cdots\wedge\alpha_{n_i}(x_{n_i}):(x_1,\ldots,x_{n_i},x)\in R_i\}\]
\vspace{0ex}

\noindent Since joins in the lattice $\prod_JL_j$ are componentwise, 

\[f_i(\alpha_1,\ldots,\alpha_{n_i})(x)(j) \,\,=\,\, \bigvee\{\alpha_1(x_1)(j)\wedge\cdots\wedge\alpha_{n_i}(x_{n_i})(j): (x_1,\ldots,x_{n_i},x)\in R_i\}\]
\vspace{0ex}

\noindent So the definition of $\Phi$ gives 

\[\Phi(f_i(\alpha_1,\ldots,\alpha_{n_i}))(j)(x) \,\,=\,\, \bigvee\{\Phi(\alpha_1)(j)(x_1)\wedge\cdots\wedge\Phi(\alpha_{n_i})(j)(x_{n_i}): (x_1,\ldots,x_{n_i},x)\in R_i\}\]
\vspace{0ex}

\noindent So the definition of $f_i^j$ gives 

\[ \Phi(f_i(\alpha_1,\ldots,\alpha_{n_i}))(j)(x) \,\,=\,\, f_i^j(\Phi(\alpha_1)(j),\ldots,\Phi(\alpha_{n_i})(j))(x)\]
\vspace{0ex}

\noindent Since this holds for each $x\in X$ we have $\Phi(f_i(\alpha_1,\ldots,\alpha_{n_i}))(j) = f_i^j(\Phi(\alpha_1)(j),\ldots,\Phi(\alpha_{n_i})(j))$. Then, since this holds for each component  $j\in J$, it follows from the fact that the operation $g_i$ of the product $\prod_JL_j^A$ is componentwise, that $\Phi(f_i(\alpha_1,\ldots,\alpha_{n_i}))=g_i(\Phi(\alpha_1),\ldots,\Phi(\alpha_{n_i}))$. 
\end{proof}

\begin{prop}
\label{reflects}
Let $\mathfrak{X}$ be a relational structure of type $\tau$ that has at least one element, and let $\phi:L\to M$ be a morphism in the category $\Lat$. 
Then $\phi$ is one-one iff $\phi^\mathfrak{X}$ is one-one, and $\phi$ is onto iff $\phi^\mathfrak{X}$ is onto. 
\end{prop}

\begin{proof}
This is a trivial consequence of the definition of $\phi^\mathfrak{X}$ in Proposition~\ref{conv}. 
\end{proof}

\begin{rmk}{\em 
It is easily seen that the categories $\Lat$ and $\Alg_\tau$ have products that are given by the usual cartesian products. So Proposition~\ref{prod} says that for a fixed relational structure $\mathfrak{X}$, the functor $\Conv(\,\cdot\,,\mathfrak{X})$ preserves products. Proposition~\ref{reflects} states that this functor $\Conv(\,\cdot\,,\mathfrak{X})$ also preserves and reflects injective and surjective maps. 
}
\end{rmk}

Before shifting focus to categories of relational structures, we discuss a modification of the convolution construction to apply to general lattices without any completeness conditions under the restriction that the relational structure $\mathfrak{X}$ is finite, or more generally, that $\mathfrak{X}$ is what we call predecessor-finite. 

\begin{defn}
A relational structure $\mathfrak{X}=(X,(R_i)_I)$ of type $\tau$ is predecessor-finite if for each $x\in X$ and each $i\in I$, the set of predecessors $\{(x_1,\ldots,x_{n_i})\in X^{n_i}:(x_1,\ldots,x_{n_i},x)\in R_i\}$ is finite. 
\end{defn}

\begin{defn}
Let $\mathfrak{X}$ be a predecessor-finite relational structure of type $\tau$. For a lattice $L$, define a lattice $L^\mathfrak{X}$ with additional operations $(f_i)_I$ of type $\tau$ by setting 
\vspace{-1ex}

\[ f_i(\alpha_1,\ldots,\alpha_{n_i})(x)=\bigvee\{\alpha_1(x_1)\wedge\cdots\wedge\alpha_{n_i}(x_{n_i}):(x_1,\ldots,x_{n_i},x)\in R_i\}\]
\vspace{-1ex}

\noindent Note that this join is a finite join since $\mathfrak{X}$ is predecessor-finite.
\end{defn}

\begin{rmk}\label{hg}{\em 
This allows for numerous small modifications to our results. For each predecessor-finite relational structure $\mathfrak{X}$ of type $\tau$, there is a functor from the category of lattices and lattice homomorphisms to the category of lattices with additional operations of type $\tau$ and their homomorphisms. This functor again preserves products, and preserves and reflects one-one and onto maps. Corresponding results hold for the convolutions of bounded lattices and predecessor-finite relational structures. 
}
\end{rmk}

We shift our focus in the consideration of categorical aspects to relational structures and the morphisms between them, the so-called p-morphisms. For a complete account, see \cite{Goldblatt}, but the essential ideas are simple. There is a categorical duality between the category of sets and functions and the category of power set Boolean algebras and the complete homomorphisms between them. This duality takes a function $p:X\to Y$ to the complete Boolean algebra homomorphism $p^{-1}$ from the power set of $Y$ to the power set of $X$. For relational structures $\mathfrak{X}$ and $\mathfrak{Y}$ of type $\tau$, the functions $p$ from $X$ to $Y$ with $p^{-1}$ giving a homomorphism from the complex algebra $\mathfrak{Y}^+$ to the complex algebra $\mathfrak{X}^+$ are exactly the $p$-morphisms from $\mathfrak{X}$ to $\mathfrak{Y}$ \cite{Goldblatt}. 

\begin{defn}
\label{p-morphism}
For relational structures $\mathfrak{X}=(X,(R_i)_I)$ and $\mathfrak{Y}=(Y,(S_i)_I)$ of type $\tau$, a function $p:X\to Y$ is a $p$-morphism if for each $i\in I$ and $x\in X$,
\vspace{-1ex}

\[ \{(y_1,\ldots,y_{n_i}):(y_1,\ldots,y_{n_i},p(x))\in S_i\}\,\,=\,\,\{(p(x_1),\ldots,p(x_{n_i})):(x_1,\ldots,x_{n_i},x)\in R_i\}\]
\vspace{-1ex}

\noindent For a type $\tau$, we let $\Rel_\tau$ be the category whose objects are the relational structures of type $\tau$ and whose morphisms are the $p$-morphisms between them. 
\end{defn}

\begin{prop}
For a bounded lattice $L$ and type $\tau$, there is a contravariant functor 
\vspace{-1ex}

\[\Conv(L,\,\cdot\,):\Rel_\tau\to\Alg_\tau\]
\vspace{-1ex}

\noindent that takes a relational structure $\mathfrak{X}$ to $L^\mathfrak{X}$, and takes a $p$-morphism $p:\mathfrak{X}\to\mathfrak{Y}$ to the homomorphism $p^L:L^\mathfrak{Y}\to L^\mathfrak{X}$ given by $p^L(\beta)=\beta\circ p$. 
\end{prop}

\begin{proof}
Clearly this assignment on objects and morphisms is well defined, produces an object of $\Alg_\tau$, and contravariantly preserves composition. It remains to show that for $p:\mathfrak{X}\to\mathfrak{Y}$ a $p$-morphism, that $p^L$ is a morphism in $\Alg_\tau$. 

Let $(\beta_j)_J$ be a family in $L^Y$. Then for $x\in X$, making use of the fact that meets in $L^X$ and $L^Y$ are componentwise and the definition of $p^L$, we have 

\[ (\bigwedge_J p^L(\beta_j))(x)\,\,=\,\, \bigwedge_J\beta_j(p(x))\,\,=\,\, (\bigwedge_J\beta_j)(p(x))\,\,=\,\,(p^L(\bigwedge_J\beta_j))(x)\]
\vspace{-1ex}

\noindent So $p^L$ preserves arbitrary meets, and similarly preserves arbitrary joins. That it preserves the bounds follows since $(p^L(0_{L^Y}))(x)=0_{L^Y}(p(x))=0_L$ for all $x\in X$, so $p^L(0_{L^Y})$ is the zero of $L^X$, with a similar argument showing that  $p^L(1_{L^Y})=1_{L^X}$. 

Suppose that $i\in I$ and that $f_i$ is the $n_i$-ary operation of $L^\mathfrak{X}$ corresponding to the $n_i+1$-ary relation $R_i$ of $\mathfrak{X}$, and that $g_i$ is the $n_i$-ary operation of $L^\mathfrak{Y}$ corresponding to the $n_i+1$-ary relation $S_i$ of $\mathfrak{Y}$. Let $\beta_1,\ldots,\beta_{n_i}\in L^Y$ and $x\in X$. The definition of a $p$-morphism in Definition~\ref{p-morphism} gives 
\vspace{-1ex}

\begin{align*}
g_i(\beta_1,\ldots,\beta_{n_i})(p(x)) 
&\,\,=\,\,\bigvee\{\beta_1(y_1)\wedge\cdots\wedge\beta_{n_i}(y_{n_i}):(y_1,\ldots,y_{n_i},p(x))\in S_i\} \\ 
&\,\,=\,\,\bigvee\{\beta_1(p(x_1))\wedge\cdots\wedge\beta_{n_i}(p(x_{n_i})):(x_1,\ldots,x_{n_i},x)\in R_i\}\\
&\,\,=\,\, f_i(p^L(\beta_1),\ldots,p^L(\beta_{n_i}))(x)
\end{align*}

\vspace{2ex}

\noindent Since this is true for each $x\in X$, we have $p^L(g_i(\beta_1,\ldots,\beta_{n_i}))=f_i(p^L(\beta_1),\ldots,p^L(\beta_{n_i}))$.
\end{proof}

\begin{rmk}{\em 
Adaptations to the functor $\Conv(\,\cdot\,,\mathfrak{X})$ were outlined in Remark~\ref{hg} depending on properties of the morphisms $\phi:L\to M$ between complete lattices chosen. Essentially, properties of $\phi$ are lifted to properties of $\phi^\mathfrak{X}$. That is not the case with the contravariant functor $\Conv(L,\,\cdot\,)$. For a $p$-morphism $p:\mathfrak{X}\to\mathfrak{Y}$, the lattice homomorphism $p^L$ is as well behaved as one could hope, preserving all joins and meets and the bounds. 
}
\end{rmk}

\begin{rmk}{\em 
Further properties of $\Conv(\,\cdot\,,\mathfrak{X})$ are given in Propositions~\ref{prod} and \ref{reflects}. It preserves products and preserves and reflects one-one and onto maps. The category $\Rel_\tau$ has coproducts given in an obvious way by union. It is easily seen that $\Conv(L,\,\cdot\,)$ takes coproducts to products, meaning 
\vspace{-1ex}

\[ L^{\bigoplus_J \mathfrak{X}_j} \,\,\simeq\,\,  \prod_JL^{\mathfrak{X}_j} \]
\vspace{0ex}

\noindent Additionally, it is easily seen that $\Conv(L,\,\cdot\,)$ takes one-one $p$-morphisms $p:\mathfrak{X}\to\mathfrak{Y}$ to onto homomorphisms $p^L:L^\mathfrak{Y}\to L^\mathfrak{X}$, and it takes onto $p$-morphisms to one-one homomorphisms. 
}
\end{rmk}

Results of this section are summarized in the following. 

\begin{thm}
\label{main}
There is a bifunctor $\Conv(\,\cdot\,,\,\cdot\,):\Lat\times\Rel_\tau\to\Alg_\tau$ that is covariant in the first argument and contravariant in the second. This functor preserves products in the first argument, and takes coproducts to products in the second. This functor preserves and reflects one-one and onto maps in the first argument. In the second argument it takes one-one maps to onto maps, and onto maps to one-one maps. 
\end{thm}

\section{Preservation of equations}

As mentioned in the introduction, there is a body of work known as correspondence theory that relates first order properties of a relational structure $\mathfrak{X}$ to equational properties of the complex algebra $\mathfrak{X}^+$. Our aim in this section is to relate the equations that are valid in a convolution algebra $L^\mathfrak{X}$ to the equations that are valid in the complex algebra $\mathfrak{X}^+$. There is a basic limitation from the outset. Correspondence theory uses the full Boolean algebra signature of the complex algebra $\mathfrak{X}^+$ as well as the additional operations of type $\tau$ from relations of $\mathfrak{X}$. In general, the lattice $L$ used to form the convolution algebra will not even have a negation, so there will be no negation inherited by the convolution algebra $L^\mathfrak{X}$. We thus restrict attention to the negation-free fragment of the language, that is, the portion formed using the binary lattice operations $\wedge,\vee$, the bounds $0,1$, and the additional operations $(f_i)_I$ for the type $\tau$. 

\begin{prop}
\label{pi}
If $L$ is a non-trivial, complete, bounded lattice, then for a relational structure $\mathfrak{X}$, the complex algebra $\mathfrak{X}^+$ is isomorphic to a subalgebra of the convolution algebra $L^\mathfrak{X}$. So any equation in the negation-free language that is valid in $L^\mathfrak{X}$ is also valid in $\mathfrak{X}^+$. 
\end{prop}

\begin{proof}
If $L$ is non-trivial, then there is an embedding $\phi$ of the 2-element lattice $2$ into $L$ that preserves bounds and finite, hence arbitrary, meets and joins. By Theorem~\ref{main}, $\phi^\mathfrak{X}$ is an embedding of $2^\mathfrak{X}$ into $L^\mathfrak{X}$, and by Proposition~\ref{curr} the complex algebra $\mathfrak{X}^+$ is isomorphic to $2^\mathfrak{X}$. 
\end{proof}

While the convolution algebra $L^\mathfrak{X}$ is defined for any complete algebra $L$ and any relational structure $\mathfrak{X}$, it is for complete lattices $L$ that are reducts of complete Heyting algebras where it enjoys its best properties. We further specialize matters temporarily. 

\begin{defn}
A spatial lattice $L$ is a bounded lattice that is isomorphic to the lattice of open sets of a topological space. 
\end{defn}

Obviously any spatial lattice is complete and distributive. Moreover, it satisfies $x\wedge\bigvee_Jy_j=\bigvee_Jx\wedge y_j$ since finite meets of open sets are given by intersections and arbitrary joins of open sets are given by unions. Complete lattices satisfying this infinite distributive law are known as frames. It is not the case that every frame is isomorphic to the open sets of a topological space. The ones that are are called spatial frames. We introduce the term spatial lattice to avoid conflict with the use of frame as a relational structure $\mathfrak{X}$. For further details, see \cite{Pultr}. 

\begin{prop}
\label{spatial}
For $L$ a non-trivial spatial lattice and $\mathfrak{X}$ a relational structure, $L^\mathfrak{X}$ and $\mathfrak{X}^+$ satisfy the same equations in the negation-free language. 
\end{prop}

\begin{proof}
Suppose that $L$ is the lattice of open sets of a topological space $(Z,\mu)$. Then, from the definition of a topological space, $L$ is a bounded sublattice of the power set $\mathcal{P}(Z)$ that is closed under arbitrary joins. By Theorem~\ref{main}, $L^\mathfrak{X}$ is isomorphic to a subalgebra of $\mathcal{P}(Z)^\mathfrak{X}$. Since the lattice $\mathcal{P}(Z)$ is isomorphic to $2^Z$, we have $L^\mathfrak{X}$ is isomorphic to a subalgebra of $(2^Z)^\mathfrak{X}$. Theorem~\ref{main} gives that $\Conv(\,\cdot\,,\mathfrak{X})$ preserves products, so $(2^Z)^\mathfrak{X}$ is isomorphic to $(2^\mathfrak{X})^Z$. So $L^\mathfrak{X}$ is isomorphic to a sublagebra of $(2^\mathfrak{X})^Z$, hence satisfies all equations in the negation-free signature that are satisfied by $2^\mathfrak{X}$. Proposition~\ref{pi} shows that all equations satisfied by $L^\mathfrak{X}$ are satisfied by $2^\mathfrak{X}$. 
\end{proof}

It is well known \cite{Pultr} that every finite distributive lattice is isomorphic to the lattice of open sets of a topological space, namely the topology of downsets of its poset of join irreducibles. This provides the following. 

\begin{cor}
\label{finite}
For $L$ a non-trivial, finite, distributive lattice, and $\mathfrak{X}$ a relational structure, $L^\mathfrak{X}$ and $\mathfrak{X}^+$ satisfy the same equations in the negation-free language. 
\end{cor}

We turn our attention to generalizing Proposition~\ref{spatial} to general frames, that is, lattices that are reducts of complete Heyting algebras. 

\begin{defn}
For a set $A$ and natural numbers $1\leq k\leq n$, the $k^{th}$ projection on $A$ is the map $\pi_k^n:A^n\to A$ defined by setting  $\pi_k^n(a_1,\ldots,a_n)=a_k$. For natural numbers $k,n$ and maps $g:A^n\to A$ and $f_1,\ldots,f_n:A^k\to A$, the generalized composite $g[f_1,\ldots,f_n]:A^k\to A$ is defined by setting 
\vspace{-2ex}

\[ g[f_1,\ldots,f_n](a_1,\ldots,a_k)\,\,=\,\, g(f_1(a_1,\ldots,a_k),\ldots,f_n(a_1,\ldots,a_k))\]
\end{defn}
\vspace{1ex}

We frequently write an element $(a_1,\ldots,a_k)\in A^k$ as $\vec{a}$. With this notation, we write the $k^{th}$ projection as $\pi_k^n(\vec{a})=a_k$ and the generalized composite as $g[f_1,\ldots,f_n](\vec{a})=g(f_a(\vec{a}),\ldots,f_n(\vec{a}))$. 

\begin{defn}
A clone $K$ on a set $A$ is a subset of $\{f\, |\, f:A^n\to A \mbox{ for some }n\in\mathbb{N}\}$ such that for each $1\leq k\leq n$ the $k^{th}$ projection $\pi_k^n$ belongs to $K$ and $K$ is closed under generalized composition. For any set of operations on $A$, there is a smallest clone on $A$ containing it. This is called the clone generated by the operations.
\end{defn}

Recall that a subset $D$ of a partially ordered set $P$ is directed if for each non-empty subset $S\subseteq D$ there is an element of $D$ that is an upper bound of this set. The following can be expressed more generally, but this is sufficient for our purposes. 

\begin{defn}
\label{nm}
A function $f:L\to M$ from a complete lattice $L$ to a complete lattice $M$ preserves directed joins if for each directed  subset $(x_j)_J$ in $L$ we have $f(\bigvee_J x_j)=\bigvee_J f(x_j)$. 
\end{defn}

A function that preserves directed joins is easily seen to be order preserving. So the joins on both sides of the equation in Definition~\ref{nm} are directed joins. We will apply this notion to operations $f:L^n\to L$ on a lattice $L$. Here we consider $L^n$ as a lattice in its own right and consider directed families $(\vec{a}_j)_J$ of elements in $L^n$. To avoid a conflict with subscripts, we write the $k^{th}$ component of $\vec{a}_j$ as $a_j^k$. Thus $\vec{a}_j=(a_j^1,\ldots,a_j^n)$. 

\begin{prop}
\label{opo}
If $L$ is a complete lattice and $f:L^n\to L$ preserves joins in each coordinate, then $f$ preserves directed joins. 
\end{prop}

\begin{proof}
For a directed family $(\vec{a}_j)_J$ in $L^n$ we have $\bigvee_J\vec{a}_j=(\bigvee_Ja_j^1,\ldots,\bigvee_Ja_j^n)$. Repeatedly applying that $f$ is additive in each coordinate we have 
\vspace{-1ex}

\[ f(\bigvee_J\vec{a}_j) \,\,=\,\, \bigvee_{j_1\in J}\cdots\bigvee_{j_n\in J}f(a_{j_1}^1,\ldots,a_{j_n}^n) \]
\vspace{0ex}

\noindent Note that a different index must be used for each component to allow cross terms. This expression is greater than or equal to $\bigvee_Jf(\vec{a}_j) = \bigvee_Jf(a_j^1,\ldots,a_j^n)$ since there the first expression is a join of a larger set of terms. Using the directedness of $(\vec{a}_j)_J$, each term in the first expression lies beneath one in the second. So the expressions are equal, and $f$ preserves directed joins. 
\end{proof}

\begin{prop}
\label{runt}
If $L$ is a complete lattice and $S$ is a set of operations on $L$ that preserve directed joins, then each member of the clone generated by $S$ preserves directed joins. 
\end{prop}

\begin{proof}
It is easily seen that the projection maps $\pi_k^n$ preserve directed joins. For $n,k\in\mathbb{N}$ suppose that $g:L^n\to L$ and $f_1,\ldots,f_n:L^k\to L$ preserve directed joins. Then for a directed family $(\vec{a}_j)_J$ in $L^k$ we have 
\vspace{-1ex}

\begin{align*}
g(f_1(\bigvee_J(\vec{a}_j),\ldots,f_n(\bigvee_J\vec{a}_j)) & \,\, = \,\, g(\bigvee_Jf_1(\vec{a}_j),\ldots,\bigvee_Jf_n(\vec{a}_j))
\end{align*}
\vspace{0ex}

\noindent Since $(\bigvee_Jf_1(\vec{a}_j),\ldots,\bigvee_Jf_n(\vec{a}_j)) = \bigvee_J(f_1(\vec{a}_j),\ldots,f_n(\vec{a}_j))$ and this is a directed join, we have 

\[g[f_1,\ldots,f_n](\bigvee_J\vec{a}_j) \,\,=\,\, \bigvee_Jg[f_1,\ldots,f_n](\vec{a}_j)\]
\vspace{-1ex}

Since the clone generated by $S$ is the closure of the union of the set $S$ with the projections under generalized composition, each member of this clone preserves directed joins. 
\end{proof}

\begin{prop}
\label{loop}
Let $L$ be a complete lattice and $X$ be a set. If $S$ is a set of operations on $L^X$ that preserve directed joins and are finitely supported in the sense of Definition~\ref{gwot}, then each member of the clone generated by $S$ preserves directed joins and is finitely supported. 
\end{prop}

\begin{proof}
Proposition~\ref{runt} shows that every member of the clone generated by $S$ preserves directed joins, and clearly the projections are finitely supported. It remains to show that if $n,k\in\mathbb{N}$ and $g:(L^X)^n\to L^X$ and $f_1,\ldots,f_n:(L^X)^k\to L^X$ preserve directed joins and are finitely supported, then the generalized composite $g[f_1,\ldots,f_n]$ is finitely supported. Let $\alpha_1,\ldots,\alpha_n\in L^X$ and set $\vec{\alpha}=(\alpha_1,\ldots,\alpha_n)$. Then let $(\vec{\delta}_j)_J$ be the family of finitely supported elements of $L^X$ that lie beneath $\vec{\alpha}$. Note that $(\vec{\delta}_j)_J$ is a directed family, so  
\vspace{-1ex}

\begin{align*}
g[f_1,\ldots,f_n](\bigvee_J\vec{\delta}_j) & \,\, = \,\, \bigvee_J g[f_1,\ldots,f_n](\vec{\delta}_j)
\end{align*}
\vspace{0ex}

\noindent Since $\vec{\alpha}=\bigvee_J\vec{\delta}_j$ we have $g[f_1,\ldots,f_n](\vec{\alpha})=\bigvee_J g[f_1,\ldots,f_n](\vec{\delta}_j)$, and therefore the generalized composite is finitely supported. 
\end{proof}

\begin{thm}
\label{eqns}
Let $L$ be the lattice reduct of a complete Heyting algebra that has at least two elements. Then for any relational structure $\mathfrak{X}$, the algebras $L^\mathfrak{X}$ and $2^\mathfrak{X}$ satisfy the same equations involving the additional operations $(f_i)_I$ and the bounded lattice operations. 
\end{thm}

\begin{proof}
Proposition~\ref{pi} shows that any equation $s\approx t$ in the bounded lattice operations and the operations $(f_i)_I$ that is valid in $L^\mathfrak{X}$ is valid in $\mathfrak{X}^+$. We must show that if $s\approx t$ is valid in the algebra $2^\mathfrak{X}$, then it is valid in $L^\mathfrak{X}$. 

Let $S$ be the operations $(f_i)_I$ of $L^\mathfrak{X}$ together with the bounded lattice operations on $L^\mathfrak{X}$, and let $K$ be the clone generated by $S$. By Proposition~\ref{operators} the operations in $S$ are complete operators, hence preserve joins in each component. So by Proposition~\ref{opo}, the operations in $S$ preserve directed joins. Proposition~\ref{finitely supported} provides that the operations in $S$ are also finitely supported. Then Proposition~\ref{loop} yields that each member of $K$ is finitely supported. In particular, $s$ and $t$ are finitely supported. 

Assume that the terms $s$ and $t$ are $n$-ary, and let $\vec{\alpha}=(\alpha_1,\ldots,\alpha_n)$ be an $n$-tuple in $L^X$. Let $(\vec{\delta}_j)_J$ be the family of $n$-tuples of finitely supported elements of $L^X$ that lie beneath $\vec{\alpha}$. Since $s$ and $t$ are finitely supported, we have 
\vspace{-1ex}

\[ s(\vec{\alpha})=\bigvee \{s(\vec{\delta}_j):j\in J\}\quad\mbox{ and }\quad t(\vec{\alpha})=\bigvee \{t(\vec{\delta}_j):j\in J\}\]
\vspace{-1ex}

\noindent To show that $s(\vec{\alpha})=t(\vec{\alpha})$, it is sufficient to show that $s(\vec{\delta}_j)=t(\vec{\delta}_j)$ for each $j\in J$. 

If $\vec{\delta}=(\delta_1,\ldots,\delta_n)$ is an $n$-tuple of finitely supported elements of $L^X$, then collectively, $\delta_1,\ldots,\delta_n$ take finitely many values in $L$. Since the bounded sublattice generated by a finite subset of bounded distributive lattice is finite, there is a finite bounded sublattice $M$ of $L$ such that each of $\delta_1,\ldots,\delta_n$ take values in $M$, hence with $\vec{\delta}$ an $n$-tuple of elements of $M^\mathfrak{X}$. 

Since there is a bounded lattice embedding of $M$ into $L$ that preserves finite, hence arbitrary joins, Theorem~\ref{main} gives that $M^\mathfrak{X}$ is a subalgebra of $L^\mathfrak{X}$ with respect to the operations in $S$. So the result $s^M(\vec{\delta})$ of evaluating the term $s$ at the $n$-tuple $\vec{\delta}$ of $M^\mathfrak{X}$ is equal to the result $s^L(\vec{\delta})$ of evaluating the term $s$ at the $n$-tuple $\vec{\delta}$ of elements of $L^\mathfrak{X}$. Since $M$ is a finite distributive lattice, Corollary~\ref{finite} provides that $M^\mathfrak{X}$ satisfies the same equations in the operations $S$ as $2^\mathfrak{X}$. So if $s\approx t$ is valid in $2^\mathfrak{X}$, then it is valid in $M^\mathfrak{X}$, and hence $s^M(\vec{\delta})=t^M(\vec{\delta})$, giving that $s^L(\vec{\delta})=t^L(\vec{\delta})$. 

So if $\vec{\delta}$ is any $n$-tuple of finitely supported elements of $L^X$, then $s(\vec{\delta})=t(\vec{\delta})$ in $L^\mathfrak{X}$. This in particular applies to each $\vec{\delta}_j$ for $j\in J$. It follows that $s(\vec{\alpha})=t(\vec{\alpha})$ in $L^\mathfrak{X}$. Since this is true for each $n$-tuple $\vec{\alpha}$ in $L^\mathfrak{X}$, we have that $s\approx t$ is valid in $L^\mathfrak{X}$.
\end{proof}

We have shown that for a non-trivial, complete, meet-continuous distributive lattice, the algebras $L^\mathfrak{X}$ and $\mathfrak{X}^+$ satisfy the same equations in the negation-free language. We show that  these conditions are necessary. Completeness is required for the definition of the convolution algebra to be sensible. We will show that distributivity is required even in the finite setting and in the fragment of the language that does not use the lattice operations.  Among  complete distributive lattices,  meet-continuity is  required  to preserve  equations  valid  in the complex algebra. 

\begin{prop}
For $L$ a complete bounded lattice and $\mathbb{Z}_2$ the 2-element group considered as a relational structure, these are equivalent. 
\vspace{1ex}

\begin{enumerate}
\item The operation of $L^{\mathbb{Z}_2}$ corresponding to addition of $\mathbb{Z}_2$ is associative
\item $L$ is distributive
\end{enumerate}
\end{prop}

\begin{proof}
Let $*$ be the operation of $L^{\mathbb{Z}_2}$ corresponding to addition $+$ of $\mathbb{Z}_2$. We consider $\mathbb{Z}_2=\{0,1\}$ and elements of $L^{\mathbb{Z}_2}$ as ordered pairs $(a_0,a_1)$ of elements of $L$. Then 

\[ [(a_0,a_1)*(b_0,b_1)] (x)\,\,=\,\, \begin{cases} \,(a_0\wedge b_0)\vee(a_1\wedge b_1) &\mbox{if } x = 0 \\ 
\,(a_0\wedge b_1)\vee(a_1\wedge b_0) &\mbox{if } x = 1 \end{cases} 
\]
\vspace{1ex}

Suppose that $(a_0,a_1), (b_0,b_1)$ and $(c_0,c_1)$ are ordered pairs of elements of $L$. We make a calculation using the common device of representing meet by juxtaposition and join by addition to increase readability. Using this notation, $(a_0,a_1)*(b_0,b_1)=(a_0b_0+a_1b_1,a_0b_1+a_1b_0)$. 

\begin{align*}
[(a_0,a_1)*(b_0,b_1)]*(c_0,c_1)
&\,\,=\,\, ((a_0b_0+a_1b_1)c_0+(a_0b_1+a_1b_0)c_1, (a_0b_0+a_1b_1)c_1+(a_0b_1+a_1b_0)c_0) \\
(a_0,a_1)*[(b_0,b_1)*(c_0,c_1)]
&\,\,=\,\,\, (a_0(b_0c_0+b_1c_1)+a_1(b_0c_1+b_1c_0), \, a_0(b_0c_1+b_1c_0)+a_1(b_0c_0+b_1c_1) )
\end{align*}
\vspace{-1ex}

If $L$ is distributive, these two expressions are equal, hence $*$ is associative in $L^{\mathbb{Z}_2}$. This can also be obtained from our general results since $a_0,a_1,b_0,b_1,c_0,c_1$ generate a finite bounded sublattice $M$ of $L$, and $M^{\mathbb{Z}_2}$ satisfies the same equations as $\mathbb{Z}^+$, and multiplication is associative in the complex algebra of any group.  Conversely, suppose that $*$ is associative in $L^{\mathbb{Z}_2}$. Let $a_0,b_0,b_1$ be arbitrary elements of $L$. Choose $a_1$ to be the 0 of $L$, and $c_0,c_1$ to both be the 1 of $L$. Then comparing the first components of each of the above expressions gives $a_0b_0+a_0b_1=a_0(b_0+b_1)$, which is the distributive law. 
\end{proof}

For the following result  we  recall  that  for a set $X$, the  largest relation  on  $X$ is  $X\times X$. We denote this  as $\nabla_X$. For  a complete lattice $L$, and $\mathfrak{X}=(X,\nabla_X)$, the  convolution algebra $L^\mathfrak{X}$ has  an additional unary operation $f$. 

\begin{prop}
\label{zui}
For $L$ a complete bounded distributive lattice, these are equivalent. 
\vspace{1ex}

\begin{enumerate}

\item $L^{(X,\nabla_X)}$ satisfies  $f(a)\wedge f(b)=f(f(a)\wedge b)$ for each set $X$

\item  $L$ satisfies  $u\wedge \bigvee_Jv_j=\bigvee_J u\wedge v_j$ hence is the lattice reduct of a complete Heyting algebra

\end{enumerate}
\end{prop}

\begin{proof}
To see that the first condition implies the second, suppose that $u\in L$ and $v_j$ $(j\in J)$ is a family of elements in $L$ indexed over a set $J$.  Consider the convolution algebra $L^{(J,\nabla_J)}$. In this convolution algebra  take  the elements $\alpha,\beta:J\to L$ defined by setting  $\alpha(j)=u$ for each $j\in J$ and $\beta(j)=v_j$ for each $j\in J$.  Note that for any $\gamma:J\to L$, that \vspace{-1ex}

$$f(\gamma)(j)=\bigvee\{\gamma(i):i\nabla_Jj\}$$
\vspace{-1ex}

\noindent So $f(\gamma)(j)=\bigvee\{\gamma(i):i\in J\}$ for each $j\in J$. In particular, $f(\gamma)$ is a constant function. We write $f(\gamma)=w$ if this constant function takes value $w$. Simple calculations give  
\vspace{-1ex}

\[f(\alpha)=u,\quad f(\beta)=\bigvee_Jv_j\quad\mbox{ and } \quad f(f(\alpha)\wedge\beta)=\bigvee_Ju\wedge v_j\]
\vspace{-1ex}

\noindent Since we have assumed that $L^{(J,\nabla_J)}$ satisfies $f(a)\wedge f(b)=f(f(a)\wedge b)$, it follows that $u\wedge\bigvee_Jv_j=\bigvee_Ju\wedge v_j$. Thus $L$ is the reduct of a complete Heyting algebra. 

For the converse, let $X$ be a set. The operation $f$ of the complex algebra $(X,\nabla_X)^+$ is given by $f(A)=\emptyset$ if $A=\emptyset$, and $f(A)=X$ otherwise.  It follows that  this complex algebra satisfies $f(f(a)\wedge b)=f(a)\wedge f(b)$. The result then follows from Theorem~\ref{eqns}.
\end{proof}

\section{Extensions}

In this section we describe a number of extensions to the method of constructing convolution algebras and the results obtained about convolution algebras. These extensions are very much in the spirit of the results previously obtained, and the proofs are similar. All these extensions are initiated by corresponding extensions to the construction of complex algebras, particularly as it is applied in applications to modal logic. The reader should see \cite{Goldblatt} for an account. We begin with a counterpart of Definition~\ref{defnoperator}. 

\begin{defn}
\label{defndualoperator}
An $n$-ary operation $f$ on a lattice $L$ is multiplicative in its $k^{th}$ component if for each finite family $(y_j)_J$ in $L$ and each $x_1,\ldots,x_{k-1},x_{k+1},\ldots,x_n$ we have 

\[ f(x_1,\ldots,x_{k-1},\bigwedge_Jy_j,x_{k+1},\ldots,x_n\}=\bigwedge_Jf(x_1,\ldots,x_{k-1},y_j,x_{k+1},\ldots,x_n) \]

\noindent In a complete lattice, an operation is completely multiplicative in its $k^{th}$ component if the same holds for an arbitrary family $(y_j)_J$. Finally, $f$ is called a dual operator if it is multiplicative in each component, and a complete dual operator if it is completely multiplicative in each component. 
\end{defn}

In modal logic, the operator $\Diamond$ is an operator and its counterpart $\Box$ is a dual operator. Both can be obtained from a relational structure. We have discussed how operators $f_i$ are obtained from a relational structure $\mathfrak{X}$ by taking relational image. We next discuss how dual operators are obtained. We temporarily introduce some unconventional terminology and notation, that of the dual complex algebra $\mathfrak{X}^{\,-}$. The reader should compare with Definition~\ref{complex}. 

\begin{defn}
Let $\mathfrak{X}=(X,(S_i)_I)$ be a relational structure of type $\tau$, and let $\mathcal{P}(X)$ be the power set of $X$.  Define an algebra $\mathfrak{X}^{\,-}=(\mathcal{P}(X),(h_i)_I)$ of type $\tau$, called the dual complex algebra of $\mathfrak{X}$, by setting for each $i\in I$ and each family of subsets $A_1,\ldots,A_{n_i}\subseteq X$ 
\vspace{-1.5ex}

\[ h_i(A_1,\ldots,A_{n_i})=\{x:\mbox{for each }(x_1,\ldots,x_{n_i},x)\in S_i\mbox{ there is }1\leq j\leq n_i \mbox{ with }x_j\in A_j\}\]
\end{defn}

It is well known \cite{Goldblatt}, and easily seen, that each of the operations $(h_i)_I$, as well as the lattice operations $\wedge,\vee$, of the dual complex algebra $\mathfrak{X}^{\,-}$ are complete dual operators. We connect these dual complex algebras with an extension of convolution algebras as follows. 

\begin{defn} 
Given a relational structure $\mathfrak{X}=(X,(S_i)_I)$ of type $\tau$ and a complete lattice $L$, define an algebra $L^{\mathfrak{X}\,-}=(L^X,(g_i)_I)$ of type $\tau$, called the dual convolution algebra of $\mathfrak{X}$ over $L$, by setting for each $i\in I$, each $\alpha_1,\ldots,\alpha_{n_i}\in L^X$ and each $x\in X$
\vspace{-1ex}

\[ g_i(\alpha_1,\ldots,\alpha_{n_i})(x)\,\,=\,\,\bigwedge\{\alpha_1(x_1)\vee\cdots\vee\alpha_{n_i}(x_{n_i}):(x_1,\ldots,x_{n_i},x)\in S_i\}\]
\end{defn}
\vspace{1ex} 

As the following result shows, the relationship between dual complex algebras and dual convolution algebras is completely analogous to the relationship between complex algebras and convolution algebras.

\begin{prop}
\label{ccc}
For a relational structure $\mathfrak{X}$, the dual convolution algebra $2^{\mathfrak{X}\,-}$ is isomorphic to the dual complex algebra $\mathfrak{X}^{\,-}$. 
\end{prop}

\begin{proof}
We show the Boolean algebra isomorphism $\phi:2^X\to \mathcal{P}(X)$ given by $\phi(\alpha)=\{x:\alpha(x)=1\}$ is an isomorphism from the dual convolution algebra $2^{\mathfrak{X}\,-}$ to the dual complex algebra $\mathfrak{X}^{\,-}$. Let $i\in I$, $\alpha_1,\ldots,\alpha_{n_i}\in 2^X$, and $x\in X$. 
\vspace{-1ex}

\begin{align*}
x\in\phi(g_i(\alpha_1,\ldots,\alpha_{n_i})) & \mbox{ iff } g_i(\alpha_1,\ldots,\alpha_{n_i})(x) = 1\\
&\mbox{ iff }  \mbox{ for each }(x_1,\ldots,x_{n_i},x)\in S_i  \mbox{ we have }\alpha_1(x_1)\vee\cdots\vee\alpha_{n_i}(x_{n_i})=1\\
&\mbox{ iff } \mbox{ for each }(x_1,\ldots,x_{n_i},x)\in S_i \mbox{ there is $1\leq j\leq n_i$ with }\alpha_j(x_j)=1\\
&\mbox{ iff } \mbox{ for each }(x_1,\ldots,x_{n_i},x)\in S_i \mbox{ there is $1\leq j\leq n_i$ with }x_j\in\phi(\alpha_j)\\
&\mbox{ iff } x\in h_i(\phi(\alpha_1),\ldots,\phi(\alpha_{n_i}))
\end{align*}
\vspace{-1ex}

\noindent The above reasoning holds also for the case when $n_i=0$. In this case, $h_i(\,) = \{x:x\not\in S_i\}$. 
\end{proof}

Results for convolution algebras have their counterparts for dual convolution algebras. For functorial matters, we require the category $\Lat^{\,-}$ of complete lattices and maps that preserve bounds, finite joins, and arbitrary meets. We summarize matters below. 

\begin{thm}
There is a bifunctor $\Conv^{\,-}:\Lat^{\,-}\times\Rel_\tau\to\Alg_\tau$ that is covariant in the first argument and contravariant in the second. This bifunctor preserves products and preserves and reflects one-one and onto maps in the first argument. In the second argument, it takes coproducts to products, one-one maps to onto maps, and onto maps to one-one maps. 
\end{thm}

Matters are best behaved when $L$ is the lattice reduct of the dual of a complete Heyting algebra, or in other words, a complete Browerian lattice. 

\begin{thm}
For $L$ a complete Browerian lattice, the operations of  $L^{\mathfrak{X}\,-}$, including the lattice meet and join, are complete dual operators. Further, if $L$ is non-trivial, then $L^{\mathfrak{X}\,-}$ and $\mathfrak{X}^{\,-}$ satisfy the same equations in the negation-free language. 
\end{thm}

One can combine the processes of forming complex algebras and dual complex algebras \cite{Goldblatt}. The type of a relational structure can be extended to an ordered pair $\tau=(\tau_1,\tau_2)$ of types, with a relational structure $\mathfrak{X}=(X,(R_i)_I,(S_j)_J)$ of this extended type being a set $X$ with two families of relations, a family $(R_i)_I$ of type $\tau_1$, and a family $(S_j)_J$ of type $\tau_2$. The complex algebra of this extended relational structure $\mathfrak{X}^*=(\mathcal{P}(X),(f_i)_I,(g_j)_J)$ consists of the power set of $X$ with two families of operations, one family $(f_i)_I$ of operators of type $\tau$ formed from the relations $(R_i)_I$, and a family $(g_j)_J$ of dual operators formed from the relations $(S_j)_J$.  

\begin{defn}
For a bounded lattice $L$ and relational structure $\mathfrak{X}=(X,(R_i)_I,(S_j)_J)$ of extended type $\tau=(\tau_1,\tau_2)$, let $L^{\mathfrak{X}*}=(L^X,(f_i)_I,(g_j)_J)$ where 
\vspace{-1ex}

\begin{align*}
f_i(\alpha_1,\ldots,\alpha_{n_i})(x)
&\,\,=\,\,\bigvee\{\alpha_1(x_1)\wedge\cdots\wedge\alpha_{n_i}(x_{n_i}):(x_1,\ldots,x_{n_i},x)\in R_i\}\\
g_i(\alpha_1,\ldots,\alpha_{n_j})(x)
&\,\,=\,\,\bigwedge\{\alpha_1(x_1)\vee\cdots\vee\alpha_{n_j}(x_{n_j}):(x_1,\ldots,x_{n_j},x)\in S_j\}
\end{align*} 
\vspace{-1ex}

\noindent Call $L^{\mathfrak{X}*}$ the convolution of the extended relational structure $\mathfrak{X}$ over $L$. 
\end{defn}

If $\mathfrak{X}$ has ordinary type $\tau$, its convolution algebra $L^\mathfrak{X}$ is the extended convolution algebra $L^{\mathfrak{X}*}$ when $\mathfrak{X}$ is considered to have extended type $(\tau,\emptyset)$, and its dual convolution algebra $L^{\mathfrak{X}-}$ is the extended convolution algebra $L^{\mathfrak{X}*}$ when $\mathfrak{X}$ is considered to have extended type $(\emptyset,\tau)$. There are natural extensions to our results for these extended convolution algebras. 

\begin{prop}
\label{der}
Let $\mathfrak{X}$ be a relational structure of extended type $\tau=(\tau_1,\tau_2)$. Then the extended convolution algebra $2^{\mathfrak{X}*}$ is isomorphic to the extended complex algebra $\mathfrak{X}^*$. 
\end{prop}

\begin{proof}
The proofs of Propositions~\ref{curr} and \ref{ccc} can be combined. 
\end{proof}

Let $\Lat^*$ be the category of bounded lattices with morphisms being maps that preserve bounds and arbitrary joins and meets. For an extended type $\tau=(\tau_1,\tau_2)$ let a $p$-morphism between relational structures $\mathfrak{X}=(X,(R_i)_I,(S_j)_J)$ and $\mathfrak{Y}=(Y,(R_i')_I,(S_j')_J)$ of this extended type be a function $p:X\to Y$ that is a $p$-morphism considered as a function from $(X,(R_i)_I)$ to $(Y,(R_i')_I)$ and from $(X,(S_j)_J)$ to $(Y,(S_j')_J)$. Then let $\Rel_\tau$ be the category of relational structures of extended type $\tau$ and the $p$-morphisms between them. Finally, let $\Alg_\tau$ be the category of algebras consisting of bounded lattices with additional families of operations of types $\tau_1$ and $\tau_2$ together with the homomorphisms between them. Combining earlier results gives the following. 

\begin{thm}
\label{lop}
For $\tau$ an extended type, there is a bifunctor $\Conv(\,\cdot\,,\,\cdot\,):\Lat^*\times\Rel_\tau\to\Alg_\tau$ that is covariant in the first argument and contravariant in the second. This bifunctor preserves products and preserves and reflects one-one and onto maps in the first argument, and takes coproducts to products and interchanges one-one and onto maps in the second argument. 
\end{thm}

The complete lattices $L$ that worked well with the convolution construction were ones that satisfied the meet continuous law: $x\wedge\bigvee_Jy_j=\bigvee_Jx\wedge y_j$, and the ones that worked well with the dual convolution construction were ones that satisfied the join continuous law: $x\vee\bigwedge_J y_j=\bigwedge_J x\vee y_j$. To work well with the extended convolution construction requires $L$ to be complete and both join and meet continuous. Rich sources of such lattices are the reducts of any complete Boolean algebra, that is, any complete Boolean lattice, and complete chains. 

\begin{prop}
\label{xw}
Let $L$ be a complete lattice that is both meet and join continuous and let $\mathfrak{X}$ be a relational structure of extended type $\tau=(\tau_1,\tau_2)$. Then the operations $(f_i)_I$ of type $\tau_1$ of the extended convolution algebra $L^{\mathfrak{X}*}$ are complete operators, the operations $(g_j)_J$ of type $\tau_2$ are complete dual operators, and the lattice operations of $L^{\mathfrak{X}*}$ are both complete operators and complete dual operators. 
\end{prop}

Equational properties of extended convolution algebras are more delicate to determine. Again, any equation in the negation-free language that is valid in $L^{\mathfrak{X}*}$, where $L$ is non-trivial, is valid in the extended complex algebra $\mathfrak{X}^*$. When considering the converse, it is necessary to have both join continuity and meet continuity of $L$ to ensure that for all extended relational structures $\mathfrak{X}$, that $L^{\mathfrak{X}*}$ satisfies the negation-free equations that hold in $\mathfrak{X}^*$. This is because these conditions are required for the convolution algebra and dual convolution algebra to satisfy all such equations. However, we do not know whether these conditions are sufficient. A useful partial result, somewhat analogous to  Proposition~\ref{spatial}, is given below.

\begin{defn}
Let $L$ be a complete lattice. Then $L$ satisfies the complete distributive law, and is called a completely distributive lattice, if for each set $J$ and each family of indexed families $a_{j,k}$ where $k\in K_j$ for each $j\in J$, 
\vspace{-1ex}

\[ \bigwedge_{j\in J}\,\,\bigvee_{k\in K_j} a_{j,k} \,\, = \,\, \bigvee_{\alpha\in\prod_JK_j}\,\,\bigwedge_{j\in J}a_{j,\alpha(j)}\]
\end{defn}

Examples of completely distributive lattices include any finite distributive lattice, any power set lattice $\mathcal{P}(X)$, and any complete chain. In fact, there is a characterization of completely distributive lattices, but this requires a further definition. 

\begin{defn}
A map $\varphi:L\to M$ between complete lattices is a complete homomorphism if it preserves arbitrary joins and arbitrary meets. We say that $M$ is a complete sublattice of $L$ if $M$ is a subset of $L$ and the identical embedding is a complete homomorphism, and that $M$ is a complete homomorphic image of $L$ if there is a complete homomorphism from $L$ onto $M$. 
\end{defn}

A complete sublattice $R$ of a power set lattice $\mathcal{P}(X)$ is called a complete ring of sets. It is a collection of subsets of $X$ that is closed under arbitrary unions and intersections. Raney \cite{Raney} has given the following characterization of completely distributive lattices. 

\begin{prop}
\label{Raney}
A complete lattice is completely distributive iff it is a complete homomorphic image of a complete ring of sets.  
\end{prop}

In conjunction with our earlier categorical results, this provides the following. 

\begin{prop}
\label{quota}
Let $L$ be a non-trivial, complete, completely distributive lattice and let $\mathfrak{X}$ be an extended relational structure. Then $L^{\mathfrak{X}*}$ and $\mathfrak{X}^*$ satisfy exactly the same equations in the negation-free language. 
\end{prop}

\begin{proof}
Apply Proposition~\ref{Raney}. There is a set $Z$, a  complete ring of sets $S$ with the identical embedding $i:S\to\mathcal{P}(Z)$ a complete homomorphism, and a complete homomorphism $\varphi:S\to L$ mapping $S$ onto $L$. By Theorem~\ref{lop}, since $i:S\to\mathcal{P}(Z)$ is a one-one map in $\Lat^*$, there is an embedding of $S^{\mathfrak{X}*}$ into $(\mathcal{P}(Z))^{\mathfrak{X}*}$, and since $\varphi:S\to L$ is an onto map in $\Lat^*$ we have that $L^{\mathfrak{X}*}$ is a homomorphic image of $S^{\mathfrak{X}*}$. Since $\mathcal{P}(Z)$ is isomorphic to $2^Z$ and the convolution functor preserves products in its first argument, we have $(\mathcal{P}(Z))^{\mathfrak{X}*}$ is isomorphic to $\prod_Z2^{\mathfrak{X}*}$. So $L^{\mathfrak{X}*}$ is a homomorphic image of a subalgebra of a product of copies of $2^{\mathfrak{X}*}$, and by Proposition~\ref{der} these copies of $2^{\mathfrak{X}*}$ are isomorphic to the extended convolution algebra $\mathfrak{X}^*$. 
\end{proof}

One would hope to extend this result and obtain an analog of Theorem~\ref{eqns} that applies when $L$ is a complete lattice that is both meet and join continuous. However, there is a problem extending the proof of Theorem~\ref{eqns} to this setting since it involves creating a clone of operations, some of which are finitely supported, and others dually finitely supported, and control of the situation is lost. This remains an open problem that we state below. 

\begin{problem}
\label{jio}
If $L$ is a non-trivial complete lattice that is both join and meet continuous, and $\mathfrak{X}$ is an extended relational structure, do the extended convolution algebra $L^{\mathfrak{X}*}$ and the extended complex algebra $\mathfrak{X}^*$ satisfy the same equations in the negation-free language? Does this hold if $L$ is a non-trivial complete Boolean lattice?
\end{problem}

Another feature can be added to relational structures and the resulting formation of complex algebras, that of a partial ordering. See \cite{Goldblatt} for details. 

\begin{defn}
For $\tau=(\tau_1,\tau_2)$ an extended type, $\mathfrak{X}=(X,\leq,(R_i)_I,(S_j)_J)$ is an ordered, extended relational structure of type $\tau$ if it is an extended relational structure of type $\tau$ with an additional partial ordering on $X$ that satisfies for each $i\in I$ and $j\in J$
\vspace{-1ex}

\begin{align*}
\mbox{if }(x_1,\ldots,x_{n_i},x)\in R_i \mbox{ and }x\leq y &\mbox{ then }(x_1,\ldots,x_{n_i},y)\in R_i\\
\mbox{if }(x_1,\ldots,x_{n_j},x)\in S_j \mbox{ and }y\leq x &\mbox{ then }(x_1,\ldots,x_{n_j},y)\in S_j
\end{align*}
\vspace{-1ex}

\noindent The relations $R_i$ are called up-closed and the relations $S_j$ are called down-closed. 
\end{defn}

Recall that a subset $A$ of a partially ordered set $(X,\leq)$ is an up-set if $x\in A$ and $x\leq y$ implies that $y\in A$, and that a subset $A$ of $(X,\leq)$ is a down-set if $x\in A$ and $y\leq x$ implies that $y\in A$. The following is found in \cite{Goldblatt}, and is not difficult to see directly. 

\begin{prop}
Let $\mathfrak{X}$ be an ordered extended relational structure. Then the collection of upsets of $X$ is a subalgebra $\mathfrak{X}^{\uu}$ of the complex algebra $\mathfrak{X}^*$ of $\mathfrak{X}$ considered as an extended relational structure. We call $\mathfrak{X}^{\uu}$ the up-set complex algebra of $\mathfrak{X}$. 
\end{prop}

Recall that for a complete lattice $L$ and poset $(X,\leq)$, the collection of order preserving functions $\alpha:X\to L$ forms a bounded sublattice of the product $L^X$ that is closed under arbitrary joins and meets. 

\begin{prop}
For $L$ a complete lattice and $\mathfrak{X}$ an ordered extended relational structure, the set $L^{\mathfrak{X}^{\uu}}$ of order preserving functions from $X$ to $L$ is a subalgebra of the extended convolution algebra $L^\mathfrak{X}$. We call $L^{\mathfrak{X}^{\uu}}$ the ordered extended convolution of $\mathfrak{X}$ over $L$. 
\end{prop}

\begin{proof}
Suppose that $i\in I$ and $\alpha_1,\ldots,\alpha_{n_i}\in L^X$. If $x,y\in X$ with $x\leq y$, then 
\vspace{-1ex}

\begin{align*}
f_i(\alpha_1,\ldots,\alpha_{n_i})(x)& \,\,=\,\, \bigvee\{\alpha_1(x_1)\wedge\cdots\wedge\alpha_{n_i}(x_{n_i}):(x_1,\ldots,x_{n_i},x)\in R_i\} 
\end{align*}
\vspace{-1ex}

\noindent Since $R_i$ is up-closed, if $(x_1,\ldots,x_{n_i},x)\in R_i$, then $(x_1,\ldots,x_{n_i},y)\in R_i$. Since $f(\alpha_1,\ldots,\alpha_{n_i})(y)$ is a join of a larger set, it follows that $f_i(\alpha_1,\ldots,\alpha_{n_i})(x)\leq f_i(\alpha_1,\ldots,\alpha_{n_i})(y)$, showing that $f_i(\alpha_1,\ldots,\alpha_{n_i})$ is order preserving. 

Let $j\in J$ and $\alpha_1,\ldots,\alpha_{n_j}\in L^X$. For $x,y\in X$ with $x\leq y$ 
\vspace{-1ex}

\begin{align*}
g_j(\alpha_1,\ldots,\alpha_{n_j})(y)& \,\,=\,\, \bigwedge\{\alpha_1(x_1)\vee\cdots\vee\alpha_{n_j}(x_{n_j}):(x_1,\ldots,x_{n_j},y)\in S_j\} 
\end{align*}
\vspace{-1ex}

\noindent Since $S_j$ is down-closed, if $(x_1,\ldots,x_{n_i},y)\in S_j$, then $(x_1,\ldots,x_{n_i},x)\in S_j$. Since $g_j(\alpha_1,\ldots,\alpha_{n_j})(x)$ is the meet of a larger set, it follows that  $g_j(\alpha_1,\ldots,\alpha_{n_j})(x)\leq g_j(\alpha_1,\ldots,\alpha_{n_j})(y)$, showing that $g_j(\alpha_1,\ldots,\alpha_{n_j})$ is order preserving. 
\end{proof}

Examining this proof shows somewhat more. The images under the operations $f_i$ and $g_j$ of any functions in $L^X$ are order preserving. We next have the expected correspondence between ordered convolution algebras and ordered complex algebras. 

\begin{prop}
For an ordered extended relational structure $\mathfrak{X}$, the ordered extended complex algebra $\mathfrak{X}^{\uu}$ is isomorphic to the ordered extended convolution algebra $2^{\mathfrak{X}^{\uu}}$.
\end{prop}

\begin{proof}
We know there is an isomorphism $\phi:2^{\mathfrak{X}*}\to\mathfrak{X}^*$ where $\phi(\alpha)=\{x:\alpha(x)=1\}$. We need only note that $\phi$ is a bijection between the subalgebras $2^{\mathfrak{X}^{\uu}}$ of order preserving functions and $\mathfrak{X}^{\uu}$ of up-sets. 
\end{proof}

Since joins and meets in the complete lattice of order-preserving functions from $X$ to $L$ agree with joins and meets in the complete lattice $L^X$, we have the following from Proposition~\ref{xw}. 

\begin{prop}
For $L$ a complete lattice that is both meet and join continuous and $\mathfrak{X}$ an ordered relational structure of extended type $\tau=(\tau_1,\tau_2)$, the operations $(f_i)_I$ of type $\tau_1$ of the ordered extended convolution algebra $L^{\mathfrak{X}^{\uu}}$ are complete operators, the operations $(g_j)_J)$ of type $\tau_2$ are complete dual operators, and the lattice operations are both complete operators and complete dual operators. 
\end{prop}

For the matter of functoriality, the definition of $p$-morphisms must be restricted. We say that a function $p:\mathfrak{X}\to\mathfrak{Y}$ between ordered extended relational structures is an order $p$-morphism if $p$ is order preserving and it is a $p$-morphism from $\mathfrak{X}$ to $\mathfrak{Y}$ considered as extended relational structures. For an extended type $\tau$ we let $\Rel_\tau^{\uu}$ be the category of ordered extended relational structures of type $\tau$ and the order $p$-morphisms between them. 

\begin{thm}
For an extended type $\tau$ there is a bifunctor $\Conv^{\uu}(\,\cdot\,,\,\cdot\,):\Lat^*\times\Rel_\tau^{\uu}\to\Alg_\tau$  taking a complete lattice $L$ and ordered extended relational structure $\mathfrak{X}$ to the ordered convolution $L^{\mathfrak{X}^{\uu}}$. 
\end{thm}

\begin{proof}
This follows from Theorem~\ref{lop} once we notice for a morphism $\phi:L\to M$ in $\Lat^*$ and an order preserving $p$-morphism $p:\mathfrak{X}\to\mathfrak{Y}$, that $\phi^\mathfrak{X}$ maps an order preserving function $\alpha:X\to L$ to an order preserving function $\phi^\mathfrak{X}(\alpha):X\to M$, and that $p^L$ maps an order preserving function $\beta:Y\to L$ to an order preserving function $p^L(\beta):X\to L$. This is because $\phi^\mathfrak{X}(\alpha)=\phi\circ\alpha$ and $p^L(\beta)=\beta\circ p$ are composites of order preserving functions. 
\end{proof}

\begin{rmk}{\em 
There are further properties of this bifunctor. In the first argument, it preserves products and preserves and reflects one-one and onto maps. To show that it preserves onto maps, for $\beta:X\to M$ order preserving consider $\alpha(x)=\bigvee\{a:\phi(a)\leq\beta(x)\}$. In the second argument, it takes coproducts to products. It is not the case that it takes one-one maps to onto maps. Consider $p$ mapping the a 2-element antichain $\mathfrak{X}$ to a 2-element chain $\mathfrak{Y}$ and the induced map $p^2:2^\mathfrak{Y}\to 2^\mathfrak{X}$. This cannot be onto since there are 3 order-preserving maps from $Y$ to $2$ and 4 order-preserving maps from $X$ to 2. It does take an order-embedding $p:\mathfrak{X}\to\mathfrak{Y}$ to an onto map; for $\alpha:X\to L$ order preserving consider $\beta(y)=\bigvee\{\alpha(x):p(x)\leq y\}$. In the second argument, onto maps are taken to one-one maps. 
}
\end{rmk}

There are two directions for further generalization. We will not develop these here, but will leave them as problems for further study. Following \cite{Goldblatt}, relational structures can be equipped also with topological structure, primarily that of Priestley spaces. For such an ordered topological extended relational structure $\mathfrak{X}$, the order-topological version of its complex algebra consists of its clopen up-sets. These clopen up-sets correspond to continuous order preserving maps from $X$ to the 2-element lattice $2$ with the discrete topology. In extending this to convolution algebras $L^\mathfrak{X}$ there are many options featuring topological structure on $L$. 

\begin{problem}
Develop an order-topological version of convolution algebras $L^\mathfrak{X}$. 
\end{problem}

A second direction involves the presence of further structure on $L$. Our results are best behaved when $L$ is the lattice reduct of a complete Heyting algebra. In this case $L$ carries a natural Heyting implication $\to$ and negation $\neg$. These lift to the full convolution algebra $L^\mathfrak{X}$ by taking the coordinatewise operations in the power $L^X$. When applied to the convolution algebra $2^\mathfrak{X}$ realizing the complex algebra, this negation is the Boolean negation that plays an important role in many of the more interesting aspects of correspondence theory. It would be of interest if portions of the correspondence theory can be developed using a Heyting implication and negation, at least in simple cases such as convolution algebras over linear algebras, or other well understood settings. 

\begin{problem}
Incorporate Heyting negation and implication into a type of correspondence theory for convolution algebras. 
\end{problem}

\section{Examples}

Here we discuss several examples  placing convolution algebras in the context of  various algebraic structures considered in extensions of classic logic. These include Heyting versions of monadic algebras and relation algebras, and the truth value algebra from type-2 fuzzy sets.  In discussing these examples, we consider more specific equations that  are preserved when forming convolution algebras. Due to the specific nature of these equations, we obtain some results that are outside the scope of the more general results on preservations of equations given before. There is surely much more to be done in this direction, but the following points to some paths. 

\begin{defn}
A unary operation $\Diamond$ on a bounded lattice $L$ is a closure operator if it is order preserving and satisfies (i) $a\leq\Diamond a$ and (ii) $\Diamond\Diamond a =\Diamond a$. A closure operator is finitely additive if it additionally satisfies  (iii) $\Diamond 0 = 0$ and (iv) $\Diamond (a\vee b)=\Diamond a \vee \Diamond b$. A unary operation $\Box$ on $L$ is an interior operator if it is order preserving and satisfies (i) $\Box a\leq a$ and (ii) $\Box\Box a=\Box a$. An interior operator is finitely multiplicative if it additionally satisfies (iii) $\Box 1 = 1$ and (iv) $\Box(a\wedge b)=\Box a\wedge \Box b$. 
\end{defn}

\begin{prop}
\label{closure}
Let $L$ be a complete distributive lattice and $\mathfrak{X}=(X,R)$ be a relational structure with a binary relation $R$.  Then  for $f$ the additional unary operation of $L^\mathfrak{X}$ we have
\vspace{1ex}

\begin{enumerate}
\item  $f(a\vee b)=f(a)\vee f(b)$ and $f(0)=0$ holds in $L^\mathfrak{X}$, so $f$ is an operator
\item  $a\leq f(a)$ holds in $L^\mathfrak{X}$ iff $R$ is reflexive
\item $f(f(a))\leq f(a)$ holds in $L^\mathfrak{X}$ iff $R$ is transitive
\end{enumerate}
\vspace{1ex}
\noindent Thus $f$ is a finitely additive closure operator  iff  $R$ is reflexive and transitive. Dually, the operation $g$ of $L^{\mathfrak{X}-}$ is a finitely multiplicative interior operator iff $R$ is reflexive and transitive.  
\end{prop}

\begin{proof}
The first statement is given by Proposition~\ref{operators}. For the forward direction of the second statement, suppose that $x\in X$. Consider the function $\alpha=\chi_{\{x\}}$ that takes value 1 at $x$ and $0$ elsewhere. Then since $\alpha(x)\leq f(\alpha)(x)$ we have that $1=\bigvee\{\alpha(y):(y,x)\in R\}$. Thus $(x,x)\in R$, showing that $R$ is reflexive. For the converse, suppose that $R$ is reflexive. Then for any function $\alpha\in L^X$ we have $\alpha(x)\leq\bigvee\{\alpha(y):(y,x)\in R\}$ since $(x,x)\in R$. Thus $\alpha\leq f(\alpha)$. For the forward direction of the third statement, let $x,y,z\in X$ with $(x,y),(y,z)\in R$. Let $\alpha$ again be the function taking value 1 at $x$ and 0 elsewhere. It is easy to see that $f(f(\alpha))(z)=1$. Since $f(f(\alpha))\leq f(\alpha)$, it follows that $f(\alpha)(z)=1$, hence $(x,z)\in R$. So $R$ is transitive. Conversely, if $R$ is transitive, then for any $\alpha\in L^X$ and any $x\in X$
\vspace{-1ex}

\[f(f(\alpha))(z)\,\,=\,\,\bigvee\{f(\alpha)(y):(y,z)\in R\}\,\,=\,\,\bigvee\{\bigvee\{\alpha(x):(x,y),(y,z)\in R\}\]
\vspace{-1ex}

\noindent Thus since $R$ is transitive, $f(f(\alpha))(z)\leq f(\alpha)(z)$, showing that $f(f(\alpha))\leq f(\alpha)$. 

This establishes that the operation $f$ of $L^\mathfrak{X}$ is a finitely additive closure operator iff $R$ is reflexive and transitive. That the operation $g$ of $L^{\mathfrak{X}-}$ is a finitely \mbox{multiplicative} interior operator iff $R$ is reflexive and transitive follows by duality. To see this, let $L^d$ be the order dual of the lattice $L$. Then $g$ is the operation on $L^X$ formed by taking the operation of $(L^d)^\mathfrak{X}$. The operation $g$ is a finitely multiplicative interior operator on $L^X$ iff it is a finitely additive closure operator on $(L^X)^d=(L^d)^X$, and this occurs iff $R$ is reflexive and transitive.
\end{proof}

We next consider the first of our specific instances, that of monadic Heyting algebras \cite{Montiero}. 

\begin{defn}
\label{monadic}
A monadic Heyting algebra is a Heyting algebra $H$ with a finitely additive closure operation $\Diamond$ and a finitely multiplicative interior operation $\Box$ that satisfy 
\vspace{1ex}

\begin{enumerate}
\item $\Box\Diamond a = \Diamond a$
\item $\Diamond\Box a = \Box a$
\item $\Diamond (\Diamond a\wedge b)=\Diamond a\wedge\Diamond b$
\end{enumerate}
\end{defn}

Monteiro and Varsavsky \cite{Montiero} introduced functional monadic Heyting algebras. These were ones constructed as follows: for a complete Heyting algebra $L$ and set $X$, define operations $\Diamond$ and $\Box$ on $L^X$ by setting 
\vspace{-1ex}

\[ \Diamond(\alpha)(x)\,\,=\,\,\bigvee\{\alpha(y):y\in X\}\quad\mbox{ and }\quad\Box(\alpha)(x)\,\,=\,\,\bigwedge\{\alpha(y):y\in X\}\]
\vspace{-1ex}

\noindent  They showed that with these operations and the natural Heyting algebra structure that $L^X$ is a monadic Heyting algebra. The following is obvious from the definitions. 

\begin{prop}
Let $L$ be a complete Heyting algebra and $X$ be a set. Let $\mathfrak{X}=(X,\nabla_X,\nabla_X)$ be the relational structure of extended type $\tau=(1,1)$ where $\nabla_X$ is the relation $X\times X$. Then Monteiro and Varsavsky's functional monadic Heyting algebra is the extended convolution algebra $L^{\mathfrak{X}*}$.
\end{prop}

\begin{rmk}{\em 
Let us consider the fact that this $L^{\mathfrak{X}*}$ satisfies the axions for monadic Heyting algebras in the context of the results on preservation of equations we have given. That $\Diamond$ is a finitely additive closure operator and $\Box$ is a finitely multiplicative interior operator are given by Proposition~\ref{closure}, and that equation~(3) of Definition~\ref{monadic} holds by Theorem~\ref{eqns} since $L$ is a Heyting algebra. The dual condition to (3) involving $\Box$ does not hold unless $L$ is a dual Heyting algebra (see Proposition~\ref{zui}). That equations (1) and (2) of Definition~\ref{monadic} hold does not follow from any results so far established, but is easily verified directly. }
\end{rmk}

In \cite{BezHard} it was shown by using amalgamation techniques that every monadic Heyting algebra is a subalgebra of a functional monadic Heyting algebra. In other words, the variety of monadic Heyting algebras is generated by the convolution algebras of the extended relational structures $(X,\nabla_X,\nabla_X)$. This can be viewed as analogous to completeness results from modal logic stating that certain varieties of modal algebras are generated by complex algebras of classes of relational structures. The key tool in such completeness results for complex algebras is the notion of canonical extensions as introduced by J\'onsson and Tarski \cite{J-T 1, J-T 2}. So far analogous questions are completely untouched for convolution algebras. We record this below as an open problem. 

\begin{problem}
Is there a procedure akin to canonical extensions for complex algebras that would provide, in some instances, results saying that certain varieties of lattices with additional operators are generated by the convolution algebras that they contain? 
\end{problem}

We next consider matters related to relation algebras. We begin with the definition of a relation algebra as given by Tarski. 

\begin{defn}
\label{relation}
A relation algebra is an algebra $(B,\wedge,\vee,\neg,0,1,;,\smile,1')$ where $(B,\wedge,\vee,\neg,0,1)$ is a Boolean algebra and 
\vspace{1ex}

\begin{enumerate}
\item $a;(b;c)=(a;b);c$
\item $a;1'=a=1';a$
\item $(a\vee b);c = (a;c)\vee(b;c)$ and $a;(b\vee c)=(a;b)\vee(a;c)$
\item $(a^\smile)^\smile = a$
\item $(a\vee b)^\smile = a^\smile\vee b^\smile$
\item $(a;b)^\smile = b^\smile;a^\smile$
\item $(a^\smile;\neg(a;b))\vee\neg b = \neg b$
\end{enumerate}
\vspace{1ex}

\noindent It is a consequence of these axioms that De Morgan's identities hold, 
\vspace{1ex}

\begin{enumerate}
\setcounter{enumi}{7}
\item $a;b\leq c\Leftrightarrow a^\smile;\neg c\leq \neg b\Leftrightarrow \neg c;b^\smile\leq \neg a$
\end{enumerate}
\end{defn}

It is well known that for a group $\mathfrak{G}=(G,\cdot,^{-1},e)$ considered as a relational structure with one ternary relation, the group multiplication $\cdot$, one binary relation $^{-1}$, and one unary relation $\{e\}$ for the group identity, that the complex algebra $\mathfrak{G}^+$ is a relation algebra. We next extend this to the convolution algebra setting. Here is the first instance where we include Heyting algebra operations in our considerations. 

\begin{prop}
For a group $\mathfrak{G}=(G,\cdot,^{-1},e)$ and complete Heyting algebra $L$, the convolution algebra $L^\mathfrak{G}$ is a Heyting algebra with pseudocomplement $\neg$ and an additional binary operation $;$, unary operation $^\smile$, and constant $1'$ that satisfies 
\vspace{1ex}

\begin{enumerate}
\item $a;(b;c)=(a;b);c$
\item $a;1'=a=1';a$
\item $(a\vee b);c = (a;c)\vee(b;c)$ and $a;(b\vee c)=(a;b)\vee(a;c)$
\item $(a^\smile)^\smile = a$
\item $(a\vee b)^\smile = a^\smile\vee b^\smile$
\item $(a;b)^\smile = b^\smile;a^\smile$
\item $(a^\smile;\neg(a;b))\vee\neg b = \neg b$
\end{enumerate}
\vspace{1ex}

\noindent The convolution algebra also satisfies the following modified form of De Morgan's identities
\vspace{1ex}

\begin{itemize}
\item[(8$^\prime$)]  $a;b\leq \neg\neg c\Leftrightarrow a^\smile;\neg c\leq \neg b\Leftrightarrow \neg c;b^\smile\leq \neg a$
\end{itemize}
\vspace{1ex}

\noindent These agree with De Morgan's identities except the first $c$ is replaced by $\neg\neg c$. Finally, this convolution algebra satisfies the original form of De Morgan's identities iff $L$ is a Boolean algebra. 
\end{prop}

\begin{proof}
That equations (1)-(6) hold in $L^\mathfrak{G}$ is an immediate consequence of Theorem~\ref{eqns} since they hold in the complex algebra $\mathfrak{G}^+$. Equation (7) and De Morgan's identities (8$^\prime$) must be considered separately since they involve the Heyting negation $\neg$. 

Equation (7) is equivalent to $a^\smile;\neg(a;b)\leq\neg b$, and by the nature of the Heyting negation, this is equivalent to $(a^\smile;\neg(a;b))\wedge b=0$. Suppose $\alpha,\beta\in L^G$ and $x\in G$. Then 

\begin{align*}
[(\alpha^\smile;\neg(\alpha;\beta))\wedge\beta](x) &\,\,=\,\, \bigvee\{\alpha^\smile(y)\wedge\neg(\alpha;\beta)(z):yz=x\}\wedge\beta(x)\\
&\,\,=\,\,\bigvee\{\alpha^\smile(y)\wedge\beta(x)\wedge\neg(\alpha;\beta)(z):yz=x\}\\
&\,\,=\,\,\bigvee\{\alpha(y)\wedge\beta(x)\wedge\neg(\alpha;\beta)(z):y^{-1}z=x\}
\end{align*}
\vspace{-1ex}

\noindent If $y^{-1}z=x$, then $yx=z$, so $\alpha(y)\wedge\beta(x)\leq(\alpha;\beta)(z)$, and therefore $\alpha(y)\wedge\beta(x)\wedge\neg(\alpha;\beta)(z)=0$. So the expression above is equal to 0, showing that equation (7) holds. 

We now consider the modified form of De Morgan's identities (8$^\prime$). We first note that due to the nature of the Heyting negation, these are equivalent to the following.  
\vspace{-1ex}

\[(a;b)\wedge\neg c = 0 \,\,\Leftrightarrow\,\, (a^\smile;\neg c)\wedge b=0\,\,\Leftrightarrow\,\, (\neg c;b^\smile)\wedge a = 0\]
\vspace{-1ex}

Let $\alpha,\beta,\gamma\in L^G$. Having $(\alpha;\beta)\wedge\neg\gamma = 0$ is equivalent to $(\alpha;\beta)(z)\wedge\neg\gamma(z) = 0$ for all $z\in G$. Using the fact that $(\alpha;\beta)(z)=\bigvee\{\alpha(x)\wedge\beta(y):xy=z\}$, an application of meet continuity gives the first of the items below. Using the fact that for $\lambda\in L^G$ we have $\lambda^\smile(u)=\lambda(u^{-1})$, the other two items follow similarly by evaluating the left side at an arbitrary $y\in G$ for  the second item, and at an arbitrary $x\in G$ for the third item. 
\vspace{-1ex}

\begin{align*}
(\alpha;\beta)\wedge\neg\gamma = 0&\,\,\Leftrightarrow\,\,  \alpha(x)\wedge\beta(y)\wedge\neg\gamma(z)=0\,\,\mbox{for all $x,y,z$ with $xy=z$}\\
(\alpha^\smile;\neg\gamma)\wedge\beta = 0&\,\,\Leftrightarrow\,\,  \alpha(x)\wedge\neg\gamma(z)\wedge\beta(y)=0\,\,\mbox{for all $x,y,z$ with $x^{-1}z=y$}\\
(\neg\gamma;\beta^\smile)\wedge\alpha = 0&\,\,\Leftrightarrow\,\,  \neg\gamma(z)\wedge\beta(y)\wedge\alpha(x)=0\,\,\mbox{for all $x,y,z$ with $zy^{-1}=x$}
\end{align*}
\vspace{-1ex}

\noindent Since $xy=z$ iff $x^{-1}z=y$ iff $zy^{-1}=x$, these statements are equivalent. 

For the further comment about $L^\mathfrak{G}$ satisfying the original form (8) of De Morgan's identities iff $L$ is Boolean, note that if $L$ is Boolean, then $\neg\neg c = c$, so $L^\mathfrak{G}$ satisfiying the modified identities implies that it satisfies the original identities. Conversely, if $L^\mathfrak{G}$ satisfies the original identities, then taking $b=1'$ we have that $L^\mathfrak{G}$ satisfies $a\leq c\Leftrightarrow \neg c\leq \neg a$. Then since $\neg a\leq \neg\neg\neg a$ holds in any Heyting algebra, this condition gives that $\neg\neg a\leq a$ holds in $L^\mathfrak{G}$, and this implies that the Heyting reduct of $L^\mathfrak{G}$ is Boolean, and this implies that $L$ is Boolean. 
\end{proof}

\begin{defn}
A binary operation $\cdot$ on a lattice $L$ is residuated if for each $a,b\in L$ there is a largest element $a\,\backslash\, b$ in $\{c:a\cdot c\leq b\}$ and a largest element $b\,/\,a$ in $\{c:c\cdot a\leq b\}$. 
\end{defn}

There has been recent interest in the study of lattices with a residuated binary operation~\cite{Ono}. It is well known that any binary operation that is completely additive in each argument, i.e. is a complete operator, is residuated. In view of Proposition~\ref{operators}, the convolution algebra $L^\mathfrak{X}$ of any relational structure $\mathfrak{X}$ with a ternary relation over a complete Heyting algebra $L$ provides a binary complete operator. So the study of convolution algebras may provide a good source of complete Heyting algebras with additional residuated operations. We next make a small example in this direction tied to our study of relation algebras. 

\begin{examp}{\em 
For a Heyting algebra $L$ and group $\mathfrak{G}$, let $0'$ be defined to be the element $\neg 1'$ in the convolution algebra $L^\mathfrak{G}$. Thus 
\vspace{-1ex}

\[1'(x)=\begin{cases} \,\,0 & \mbox{if }x\neq e\\ \,\,1 & \mbox{if }x=e\end{cases} \quad\quad \mbox{ and }\quad\quad
0'(x)=\begin{cases} \,\,1 & \mbox{if }x\neq e\\ \,\,0 & \mbox{if }x=e\end{cases}\]
\vspace{0ex}

\noindent Note that $\alpha;\gamma\leq 0'$ iff $(\alpha;\gamma)(e)=0$. Since $(\alpha;\gamma)(e)=\bigvee\{\alpha(x^{-1})\wedge\gamma(x):x\in G\}$ it follows that $(\alpha;\gamma)\leq 0'$ iff $\gamma(x)\leq\neg\alpha(x^{-1})$ for each $x\in G$, which occurs iff $\gamma\leq\neg(\alpha^\smile)$. Applying similar reasoning to $\gamma;\alpha\leq 0'$ gives 
\vspace{-1ex}

\[ \alpha\,\backslash\, 0' \,\, = \,\, \neg(\alpha^\smile) \,\, = \,\, 0'\, / \, \alpha\]
\vspace{-1ex}

\noindent So the convolution algebra is a Heyting algebra with residuated monoidal operation $;$ with unit $1'$ and constant $0'$. Using the easily verified fact that $\neg(a^\smile)=(\neg a)^\smile$ in the convolution algebra, it additionally satisfies 
\vspace{-1ex}

\[(0'\,/\,a)\,\backslash\, 0'  = \neg\neg a = 0'\,/\,(a\,\backslash\, 0') \quad\mbox{ and }\quad 0'\,/\,a=a\,\backslash\,0'\]
\vspace{-1ex}

\noindent This is related to structures called bounded \textsc{gbi}-algebras By Galatos and Jipsen \cite{Jipsen}, but their \textsc{gbi}-algebras satisfy $(0'\,/\,a)\,\backslash\, 0'  = a = 0'\,/\,(a\,\backslash\, 0')$. So the convolution algebra $L^\mathfrak{G}$ is a \textsc{gbi}-algebra iff $L$ is Boolean. 
}
\end{examp}

Our final example of a convolution algebra is the one that originated our interest in the topic, the algebra of truth values of type-2 fuzzy sets as introduced by Zadeh \cite{Zadeh1,Zadeh2}. In the terminology of the current paper, it becomes the following. 

\begin{defn}
Let $\I = [0,1]$ be real the unit interval. We consider $\I$ as a bounded lattice, and also consider $\mathfrak{I}=(\I,\wedge,\vee,\neg,0,1)$ as a relational structure with two binary operations, one unary operation, and two nullary operations, hence with two ternary relations, one binary relation, and two unary relations. The truth value algebra for type-2 fuzzy sets is the convolution algebra $\I^\mathfrak{I}$ that is $(\I^{\I},\sqcap,\sqcup,*,1_0,1_1)$ where $\sqcap$ and $\sqcup$ are convolutions of $\wedge$ and $\vee$, $*$ is the convolution of $\neg$, and $1_0$ and $1_1$ are the convolutions of $0,1$. 
\end{defn}

In \cite{HWW}, and in many other papers referenced there, basic properties of the truth value algebra are developed. The current techniques outlined in this note not only encompass many of these, but also open the path to further results.  Suppose that $\mathfrak{J}$ is any extended relational structure over $\I$ with operations $\wedge,\vee,\neg,0,1$ and also perhaps including t-norms and co-norms. Then consider an extended convolution algebra $\I^{\mathfrak{J}*}$ where some of these operations are convoluted using joins and others with meets. Since $\I$ is a complete chain, hence a completely distributive lattice, we may apply Proposition~\ref{quota} to obtain that the negation-free equations valid in $\I^{\mathfrak{J}*}$ are exactly those valid in the extended complex algebra $\mathfrak{J}^*$. So there are powerful tools to study properties of mixed convolutions of the operations of $\I=(\I,\wedge,\vee,\neg,0,1)$ as well as mixed convolutions of various t-norms and co-norms and also of other relations on $\I$.

\end{document}